\documentclass[12pt]{amsart}
\usepackage{amscd, amsfonts, amssymb}
\usepackage{fullpage}
\usepackage{latexsym}
\usepackage{verbatim}
\usepackage{enumerate}
\usepackage[colorlinks=true,citecolor=red,linkcolor=blue]{hyperref}
\usepackage{tikz}
\usetikzlibrary{arrows,matrix,decorations,positioning, backgrounds,fit}

\newcommand\cc{\mathfrak c}
\renewcommand\gg{\mathfrak g}
\newcommand\gl{\mathfrak{gl}}
\newcommand\pgl{\mathfrak{pgl}}
\newcommand\hh{\mathfrak h}

\newcommand\zz{\mathfrak z}

\newcommand\inverse{{^{-1}}}

\newcommand\ra{\rightarrow}

\newcommand\OO{\mathcal O}

\newcommand{\FF}{{\mathbb F}}
\newcommand{\GG}{{\mathbb G}}
\newcommand{\NN}{{\mathbb N}}
\newcommand{\ZZ}{{\mathbb Z}}
\newcommand{\QQ}{{\mathbb Q}}
\newcommand{\RR}{{\mathbb R}}

\renewcommand{\AA}{{\mathbb A}}

\newcommand{\tuple}[1]{{\mathbf {#1}}}

\newcommand{\ccc}[1]{{\overline{#1}^{\rm c}}}

\DeclareMathOperator{\Char}{char}

\DeclareMathOperator{\GL}{GL}
\DeclareMathOperator{\Gal}{Gal}
\DeclareMathOperator{\SL}{SL}

\DeclareMathOperator{\PGL}{PGL}

\DeclareMathOperator{\Lie}{Lie}

\DeclareMathOperator{\IM}{Im}
\DeclareMathOperator{\End}{End}

\DeclareMathOperator{\Mat}{Mat}

\numberwithin{equation}{section}

\newtheorem{thm}[equation]{Theorem}

\newtheorem{lem}[equation]{Lemma}
\newtheorem{cor}[equation]{Corollary}
\newtheorem{prop}[equation]{Proposition}

\theoremstyle{definition}
\newtheorem{defn}[equation]{Definition}
  
\newtheorem{ex}[equation]{Example}
\newtheorem{exs}[equation]{Examples}
\theoremstyle{remark}
\newtheorem{rem}[equation]{Remark}
\theoremstyle{remark}
\newtheorem{rems}[equation]{Remarks}

\newtheorem{qn}[equation]{Question}

\newcommand{\ovl}{\overline}


\title[Cocharacter-closure and the rational Hilbert-Mumford Theorem]
{Cocharacter-closure and the rational Hilbert-Mumford Theorem}

\author[M.\  Bate]{Michael Bate}
\address
{Department of Mathematics,
University of York,
York YO10 5DD,
United Kingdom}
\email{michael.bate@york.ac.uk}

\author[S. Herpel]{Sebastian Herpel}
\address
{Fakult\"at f\"ur Mathematik,
Ruhr-Universit\"at Bochum,
D-44780 Bochum, Germany}
\email{sebastian.herpel@rub.de}

\author[B.\ Martin]{Benjamin Martin}
\address
{Department of Mathematics,
University of Aberdeen,
King's College,
Fraser Noble Building,
Aberdeen AB24 3UE,
United Kingdom}
\email{b.martin@abdn.ac.uk}

\author[G. R\"ohrle]{Gerhard R\"ohrle}
\address
{Fakult\"at f\"ur Mathematik,
Ruhr-Universit\"at Bochum,
D-44780 Bochum, Germany}
\email{gerhard.roehrle@rub.de}

\begin{document}

\begin{abstract}
For a field $k$, let $G$ be a reductive $k$-group and
$V$ an affine $k$-variety on which $G$ acts.
Using the notion of cocharacter-closed $G(k)$-orbits in $V$,  
we prove a rational version of the celebrated 
Hilbert-Mumford Theorem from geometric invariant theory.
We initiate a study of applications 
stemming from this rationality tool.
A number of examples are discussed to illustrate the concept of 
cocharacter-closure and to highlight 
how it differs from the usual Zariski-closure.
\end{abstract}

\maketitle

\setcounter{tocdepth}{1}
\tableofcontents

\section{Introduction}
\label{sec:intro}

Let $G$ be a (possibly non-connected) reductive algebraic group over an algebraically closed field $k$, and let $V$ be an affine variety over $k$ on which $G$ acts.  A central problem in geometric invariant theory is to understand the structure of the set of orbits of $G$ on $V$.  The closed orbits are particularly important, because they can be identified with the points of the quotient variety $V/\!/G$.
Given $v\in V$, it is well known that the closure $\overline{G\cdot v}$ of the orbit $G\cdot v$ contains a unique closed orbit ${\mathcal O}$.  The Hilbert-Mumford Theorem says there exists a cocharacter $\lambda$ of $G$ such that the limit $\lim_{a\to 0} \lambda(a)\cdot v$ exists and lies in ${\mathcal O}$---in fact, we can replace ${\mathcal O}$ with any closed $G$-stable subset of $V$ that meets $\overline{G\cdot v}$ \cite[Thm.\ 1.4]{kempf}.  This gives a characterization of the closed orbits in terms of cocharacters: if the orbit $G\cdot v$ is not closed, then there is a cocharacter $\lambda$ of $G$ such that $\lim_{a\to 0} \lambda(a)\cdot v$ exists and lies outside $G\cdot v$.  Conversely, if $G\cdot v$ is closed then $\lim_{a\to 0} \lambda(a)\cdot v$ lies in $G\cdot v$ for all $\lambda$ such that the limit exists (cf.\ Section~\ref{subsec:Gvars}).

A strengthening of the Hilbert-Mumford Theorem due to Hesselink \cite{He}, Kempf \cite{kempf} and Rousseau \cite{rousseau} 
shows that if $G\cdot v$ is not closed, then there is a class of so-called ``optimal'' cocharacters $\lambda_{\rm opt}$ such that the limit
$\lim_{a\to 0} \lambda_{\rm opt}(a)\cdot v$ exists in $V$ but lies outside $G\cdot v$.  Each cocharacter $\lambda_{\rm opt}$ enjoys some nice properties: for instance, if $G$ is connected, then the parabolic subgroup $P_{\lambda_{\rm opt}}$ associated to 
$\lambda_{\rm opt}$ contains the stabiliser $G_v$.  Moreover, if $G$, $V$ and the $G$-action are defined over a perfect 
subfield $k_0$ of $k$, $k/k_0$ is algebraic and $v\in V(k_0)$, then $\lambda_{\rm opt}$ is ${\rm Gal}(k/k_0)$-fixed 
and hence is defined over $k_0$.  
The (strengthened) Hilbert-Mumford Theorem has become an indispensable tool 
in algebraic group theory and has numerous 
applications in geometric invariant theory and beyond \cite{mumford}: e.g., geometric complexity theory \cite{MS1}, \cite{MS2}, nilpotent and unipotent elements of reductive groups \cite{CP}, \cite{Premet2}, \cite{Premet1}, \cite{bate}, moduli spaces of bundles \cite{Gomez}, good quotients in geometric invariant theory \cite{Hausen}, Hilbert schemes \cite{Nakajima}, moduli spaces of sheaves \cite{HL}, the structure of the Horn cone \cite{Ressayre}, K\"ahler geometry \cite{Szekelyhidi}, filtrations for representations of quivers \cite{Zamora}, symplectic quotients \cite[App.~2C]{mumford}, degenerations of modules \cite{zwara} and $G$-complete reducibility \cite{BMR}, \cite{GIT}.

Now suppose $k$ is an arbitrary field, not necessarily algebraically closed.
The orbit $G\cdot v$ is a union of $G(k)$-orbits.
The structure of this set of $G(k)$-orbits can be very intricate.
For instance, if $w\in G\cdot v$ and $v,w$ are $k$-points then one can ask
whether $w$ is $G(k)$-conjugate to $v$.
The answer is no in general; if $k$ is perfect then this is controlled by the Galois 1-cohomology of $G_v(k_s)$ (see Remark \ref{rem:strongerrationality}(iv) or \cite{berhuy:galois}). 
Things only get more complicated when one considers the $G(k)$-orbits that are contained in $\overline{G\cdot v}$.  Orbits of actions of reductive groups over non-algebraically closed fields have come under increasing attention, particularly from number theorists.
For instance, suppose $k$ is a global function field, let $v\in V(k)$ and let $C$ be the set of all $w\in V(k)$ such that $w$ is $G(k_\nu)$-conjugate to $v$ for every completion $k_\nu$ of $k$; then Conrad showed that $C$ is a finite union of $G(k)$-orbits \cite[Thm.~1.3.3]{conrad}.
Bremigan studied the strong topology of the orbits when $k$ is a local field \cite{Bremigan}
(see also Remark \ref{rem:strongerrationality}(v) below).
Now let $V$ be the Lie algebra $\gg$ with the adjoint action of $G$. 
When $k$ is perfect, J.~Levy  proved that if $x\in \gg$, then any two elements of the form $\lim_{a\to 0} \lambda(a)\cdot x$ that are semisimple are $G(k)$-conjugate
\cite{levy1, levy2}. 
We extend Levy's result below, see Remark \ref{rem:Levy}(ii).
W.~Hoffmann asked a question about limits $\lim_{a\to 0} \lambda(a)\cdot x$ 
when $k$ is a global field of arbitrary characteristic,
see Remark \ref{rems:hoff}(ii).  
Both Levy and Hoffmann were motivated by constructions involving orbital integrals and the Selberg trace formula.

In this paper we develop a theory of $G(k)$-orbits for $k$ an arbitrary field, building on
earlier work of three of us with Tange \cite{GIT}.  To make our results as general as possible,
we do not endow the field $k$ with any extra structure.  Moreover, we need not always assume that $v$ is a $k$-point.  
The most difficult problems arise when $k$ is non-perfect.
However, even for groups over perfect fields our methods
apply; see Theorem \ref{thm:anisotropy}(ii).
Observe that for perfect fields the hypotheses on the stabilizers
needed for our main results hold automatically 
(e.g., see Proposition  \ref{prop:perfectorseparable}).

As noted above, the concept of a closed orbit is fundamental in geometric invariant theory over algebraically closed fields.  A first problem is to devise a suitable analogue of this idea for $G(k)$-orbits.  One can define the notion of a $k$-orbit over $G$ \cite[10.2, Def.\ 4]{neron}, or study the Zariski closure of a $G(k)$-orbit, but such constructions do not appear to be helpful here (cf.\ the discussion in \cite[Rem.\ 3.9]{GIT}).  Instead we adopt an approach involving cocharacters.  Let $Y_k(G)$  denote the set of $k$-defined cocharacters of $G$.  In \cite[Def.\ 3.8]{GIT}, we made the following definition.

\begin{defn}
\label{defn:cocharclosedorbit}
The orbit $G(k)\cdot v$ is \emph{cocharacter-closed over $k$} 
provided for all $\lambda\in Y_k(G)$,
if $v':= \lim_{a\to 0} \lambda(a)\cdot v$ exists, then $v'\in G(k)\cdot v$.
\end{defn}

We now extend this definition to 
cover arbitrary subsets of $V$, 
and introduce the cocharacter-closure of a subset of $V$:
\begin{defn}
\label{defn:cocharclosure}
(a) Given a subset $X$ of $V$, we say that $X$ 
is \emph{cocharacter-closed (over $k$)} if for every $v \in X$ and
every $\lambda \in Y_k(G)$ such that $v':=\lim_{a \to 0} \lambda(a) \cdot v$ 
exists, $v' \in X$.
Note that this definition coincides with 
the one above if $X = G(k) \cdot v$ for some $v \in V$.\\
(b) Given a subset $X$ of $V$, we define the 
\emph{cocharacter-closure of $X$ (over $k$)}, denoted $\ccc{X}$,
to be the smallest subset of $V$
such that $X \subseteq \ccc{X}$ and $\ccc{X}$ is cocharacter-closed over $k$.  (This makes sense because the intersection of cocharacter-closed subsets is clearly cocharacter-closed.)
\end{defn}

It follows from the Hilbert-Mumford Theorem that 
$G\cdot v$ is cocharacter-closed over $\ovl{k}$ if and only if $G\cdot v$ is closed.  
It is obvious that $\ccc{G\cdot v}$ is contained in $\ovl{G\cdot v}$.  
Note, however, that this containment can be proper: e.g., see Example \ref{ex:g2}.

Our first main result is a rational version of the Hilbert-Mumford Theorem.

\begin{thm}
\label{thm:rat_HMT}
Let $v \in V$.  Then there is a unique cocharacter-closed $G(k)$-orbit $\OO$
inside $\ccc{G(k)\cdot v}$.  Moreover, there exists $\lambda\in Y_k(G)$
such that $\lim_{a\to 0} \lambda(a)\cdot v$ exists and lies in $\OO$.
\end{thm}

By a standard fact, the closure of a geometric 
$G$-orbit is again a union of $G$-orbits, \cite[I 1.8 Prop.]{borel}.
Thanks to Lemma \ref{lem:transitiveclosure}(i), 
the rational counterpart holds for the
cocharacter-closure 
of a $G(k)$-orbit in $V$. Therefore, 
we can mimic the usual ``degeneration'' partial order on the $G$-orbits in $V$
in this rational setting:

\begin{defn}
\label{defn:preorder}
Given $v,v' \in V$, 
we write $G(k) \cdot v'\prec G(k) \cdot v$ if $v' \in \ccc{G(k) \cdot v}$.
\end{defn}
Then it is clear that $\prec$ is reflexive and transitive, so $\prec$ gives a preorder on the set of $G(k)$-orbits in $V$.
In general, the behavior of the $G(k)$-orbits can be quite pathological;
e.g., see Example \ref{ex:rsquares}, Remark \ref{rem:strongerrationality}(i),
Example \ref{ex:SL2,GL2}(ii) and Remark \ref{rems:hoff}(ii).

Our second main result holds under some mild hypotheses on the 
stabilizer $G_v$ of $v$ in $G$.

\begin{thm} 
\label{thm:Gvkdef}
Let $v \in V$ and suppose that $G_v$ is $k$-defined. Then the following hold:
\begin{itemize}
\item[(i)] If $G \cdot v$ is Zariski-closed, then $G(k)\cdot v$ is cocharacter-closed
over $k$.
\item[(ii)]
Let $k'/k$ be an algebraic field extension and suppose that $G(k')\cdot v$
is cocharacter-closed over $k'$. Then $G(k)\cdot v$ is cocharacter-closed over $k$.
Moreover, the converse holds provided that $v \in V(k)$ and $k'/k$ is separable.
\item[(iii)]
Let $S$ be a $k$-defined torus of $G_v$ and set $L=C_G(S)$. Then
$G(k) \cdot v$ is cocharacter-closed over $k$ if and only if
$L(k) \cdot v$ is cocharacter-closed over $k$.
\item[(iv)] 
Let $w \in V$ and suppose that both
$G(k) \cdot w \prec G(k) \cdot v$ and $G(k) \cdot v \prec G(k) \cdot w$.
Then $G(k) \cdot v = G(k) \cdot w$.
\end{itemize}
\end{thm}

In fact, we prove stronger versions of the results of Theorem \ref{thm:Gvkdef}
in Proposition~\ref{prop:galois_descent}, Theorems \ref{thm:descent+ascent} and \ref{thm:MCkdefinedstabilizer} and
Corollary \ref{cor:MCkdefinedstabilizer} below.
Note that 
the rationality condition on the centralizer $G_v$ 
in Theorem \ref{thm:Gvkdef} 
is satisfied in many instances, e.g., if $v\in V(k)$ and $k$ is perfect (see Proposition \ref{prop:perfectorseparable}).

Recall that $G$ is $k$-anisotropic provided $Y_k(G) = \{0\}$.  Part (i) of the next theorem gives a characterization of $k$-anisotropic reductive groups over an arbitrary field $k$ in terms of cocharacter-closed orbits.  In the special case when $k$ is perfect, we recover in part (ii) a result of Kempf \cite[Thm.\ 4.2]{kempf}.

\begin{thm} 
\label{thm:anisotropy}
\begin{itemize}
\item[(i)] 
$G$ is $k$-anisotropic if and only if for every $k$-defined affine $G$-variety $W$ and every 
$w \in W(k)$, the orbit $G(k) \cdot w$ is cocharacter-closed over $k$.
\item[(ii)] 
Suppose $k$ is perfect. Then $G$ is $k$-anisotropic if and only if for every
$k$-defined affine $G$-variety $W$ and every  
$w \in W(k)$, the orbit $G \cdot w$ is closed in $W$.
\end{itemize}
\end{thm}

Part (ii) of Theorem \ref{thm:anisotropy} follows from part (i), 
Theorem \ref{thm:Gvkdef}(ii) and the Hilbert-Mumford Theorem.
Observe that it suffices to consider the case when $W$ is 
a $k$-defined rational $G$-module,
cf.~Remark \ref{rem:linear}.
Characterizing anisotropy over perfect fields 
in terms of closed orbits
was a question of Borel, \cite[Rem.\ 8.8 (d)]{Borel1}.  
As noted above, this question was answered by Kempf.
Birkes \cite{Birkes} proved Theorem \ref{thm:anisotropy}(ii) over
the reals and number fields, and Bremigan \cite{Bremigan} proved 
it for $p$-adic fields.
Note that Theorem \ref{thm:anisotropy}(ii) fails for 
non-perfect fields; see Remark \ref{rem:anisotropy}(ii).

\smallskip
The notion of a cocharacter-closed $G(k)$-orbit has already
proved very useful in the context of
Serre's notion of $G$-complete reducibility over $k$,
see Section \ref{sec:Gcr}.
In \cite[Thm.\ 5.9]{GIT} we gave a
geometric characterisation of the latter
using the former.
In Theorem~\ref{thm:crvscocharclosed}, we 
strengthen this result by removing the 
connectedness assumption on $G$ from \cite[Thm.\ 5.9]{GIT}.

In Corollary \ref{cor:prettycrdescent}, we prove a general 
Galois descent result for $G$-complete reducibility for arbitrary
algebraic field extensions $k'/k$ under very mild assumptions
on $\Char(k)$ which guarantee smoothness 
of centralizers of subgroups, \cite[Thm.\ 1.1]{herpel}.
General results of this nature were previously only known 
when both fields are algebraically closed or
perfect, \cite[Thms.~5.3, 5.8]{BMR}, or else when the 
extension $k'/k$ is separable, \cite[Thm.~5.11]{GIT}.

\smallskip
The paper is organized as follows.
We spend most of Section \ref{sec:prelims} recalling some results from \cite{GIT}.
In Section \ref{sec:cocharclosure},  we 
discuss the concept of cocharacter-closure in detail, and 
introduce the notion of accessibility of $G(k)$-orbits
and its relation to the cocharacter-closure of a $G(k)$-orbit,
see Lemma \ref{lem:transitiveclosure}.

Theorem~\ref{thm:rat_HMT} is proved in Section \ref{sec:rationalHMT} (Theorem~\ref{thm:uniqueaccessible}).
This is followed in Section \ref{sec:up_and_down} by a discussion of 
various ascent/descent results for field extensions on the one hand
and for Levi subgroups on the other:
see Theorems \ref{thm:leviascentdescent} and \ref{thm:descent+ascent}, 
which prove Theorem~\ref{thm:Gvkdef}(ii) (second assertion) and (iii).
A technical result needed in the proof
of Theorem \ref{thm:descent+ascent} is postponed to 
Theorem \ref{thm:k-point} in Section
\ref{sec:k-point}.

Section \ref{sec:mainconj} addresses 
questions of geometric $G$-conjugacy
versus rational $G$-conjugacy.
In particular, here we prove Theorem \ref{thm:Gvkdef}(i), (ii) (first assertion) and (iv) (see Corollary~\ref{cor:MCkdefinedstabilizer}).
We close the section with a discussion 
of when the preorder $\prec$ from Definition \ref{defn:preorder}
is a partial order.
It turns out that this is closely related to 
a descent result for geometric conjugacy, 
see Corollary \ref{cor:anti}. 
In Section \ref{sec:reduction}, we give an application of our results from Sections~\ref{sec:up_and_down} and \ref{sec:mainconj}: we show that Galois ascent
and certain conjugacy results
hold for general $G$ provided they hold for $\GL_n$.

In Section \ref{sec:Gcr}, we discuss applications to Serre's
notion of $G$-complete reducibility over $k$.
Under some mild rationality assumptions we obtain 
corresponding 
Galois and Levi ascent/descent results, see
Corollary \ref{cor:descent+ascentGCR}.

In Section \ref{sec:GLn}, we 
discuss the notion of cocharacter-closure 
in the classical context of conjugacy of 
endomorphisms under the general linear group in great detail.
Here the preorder from Definition \ref{defn:preorder}
is automatically antisymmetric (Corollary~\ref{cor:GL}).  We then consider modules over finitely generated $k$-algebras.  Our methods give a geometric way to explore the notion of degenerations of modules over a non-algebraically closed field.

We finish (Section~\ref{sec:ex}) with some 
examples of unipotent classes in $G_2$ 
and a representation of $\SL_2$,
demonstrating that the notions of Zariski-closure and
cocharacter-closure already differ for an algebraically 
closed field.

\section{Preliminaries}
\label{sec:prelims}

\subsection{Basic notation}
\label{ssec:basic}

Let $k$ be a field, let $\ovl k$
denote a fixed algebraic closure, and
let $k_s \subseteq \ovl k$ denote the separable closure
and $k_i \subseteq \ovl k$ the purely inseparable closure of $k$.
Note that $k_s=\ovl{k}$ if $k$ is perfect.
We denote the Galois group $\Gal(k_s/k)=\Gal(\ovl k/k)$ by $\Gamma$.
We use the notion of a $k$-scheme from \cite[AG.11]{borel}: a $k$-scheme is a
$\ovl k$-scheme together with a $k$-structure.
So $k$-schemes are assumed to be of finite type and reduced
separated $k$-schemes are called $k$-varieties.
Furthermore, a subscheme of a scheme $V$ over $k$ or over
$\ovl{k}$ is always a subscheme of $V$
as a scheme over $\ovl{k}$ and points of $V$ are always
closed points of $V$ as a
scheme over $\ovl{k}$.  By ``variety'' we mean ``variety over $\ovl{k}$''.

Now let $V$ be a $k$-variety.  If $k_1/k$ is an algebraic extension,
then we write $V(k_1)$ for the set of $k_1$-points of $V$.
If $W$ is a subvariety of $V$, then we set $W(k_1) = W(\ovl{k}) \cap V(k_1)$.
Here we do not assume that $W$ is $k$-defined, so $W(k_1)$
can be empty even when $k_1= k_s$.  The Galois group $\Gamma$
acts on $V(k_s)$; see, e.g., \cite[11.2]{springer}.
Recall the Galois criterion for a closed subvariety $W$ of $V$
to be $k$-defined: $W$ is $k$-defined if and only if
$W(\ovl{k})$ contains a
$\Gamma$-stable subset of $V(k_s)$ which is dense in $W$
(see \cite[Thm.~AG.14.4]{borel}).

\subsection{Algebraic groups}
\label{ssec:lags}

All linear algebraic groups are assumed to be smooth.  Let $H$ be a $k$-defined linear algebraic group.  By a subgroup of $H$ we mean a smooth subgroup.  We let $Z(H)$ denote the centre of $H$ and $H^0$ the connected component of
$H$ that contains $1$.  Recall that $H$ has a $k$-defined maximal torus \cite[18.2(i)~Thm.]{borel}.
For $K$ a subgroup of $H$, we denote the
centralizer of $K$ in $H$ by $C_H(K)$.

For the set of cocharacters (one-parameter subgroups) of $H$ we write $Y(H)$;
the elements of $Y(H)$ are the homomorphisms from the multiplicative group
$\ovl{k}^*$ to $H$. We denote the set of $k$-defined
cocharacters by $Y_k(H)$.
There is a left action of $H$ on $Y(H)$ given by
$(h\cdot \lambda)(a) = h\lambda(a)h^{-1}$ for
$\lambda\in Y(H)$, $h\in H$ and $a \in \ovl{k}^*$.
The subset $Y_k(H)$ is stabilized by $H(k)$.

The \emph{unipotent radical} of $H$ is denoted $R_u(H)$; it is the maximal
connected normal unipotent subgroup of $H$.
The algebraic group $H$ is called \emph{reductive} if $R_u(H) = \{1\}$.
Note that we allow a reductive group to be non-connected.

\subsection{Reductive groups}
\label{subsec:noncon}

Throughout the paper, $G$ denotes a $k$-defined
reductive algebraic group, possibly disconnected.
The crucial idea which allows us to deal with non-connected groups is the introduction of so-called
\emph{Richardson parabolic subgroups} (\emph{R-parabolic subgroups}) of a
reductive group $G$.
We briefly recall the main definitions and results;
for more details and further results,
the reader is referred to \cite[Sec.\ 6]{BMR} and \cite[Sec.~2.2]{GIT}.

\begin{defn}
\label{defn:rpars}
For each cocharacter $\lambda \in Y(G)$, let
$P_\lambda = \{ g\in G \mid \underset{a \to 0}{\lim}\,
\lambda(a) g \lambda(a)\inverse \textrm{ exists} \}$
(see Section~\ref{subsec:Gvars} for the definition of limit).
Recall that a subgroup $P$ of $G$ is \emph{parabolic}
if $G/P$ is a complete variety.
The subgroup $P_\lambda$ is parabolic in this sense, but the converse
is not true in general.
If we define
$L_\lambda = \{g \in G \mid \underset{a \to 0}{\lim}\,
\lambda(a) g \lambda(a)\inverse = g\}$,
then $P_\lambda = L_\lambda \ltimes R_u(P_\lambda)$,
and we also have
$R_u(P_\lambda) = \{g \in G \mid \underset{a \to 0}{\lim}\,
\lambda(a) g \lambda(a)\inverse = 1\}$.
The subgroups $P_\lambda$ for $\lambda \in Y(G)$
are called the \emph{Richardson parabolic} (or \emph{R-parabolic}) \emph{subgroups}
of $G$.
Given an R-parabolic subgroup $P$,
a \emph{Richardson Levi} (or \emph{R-Levi})
\emph{subgroup} of $P$ is any subgroup $L_\lambda$
such that $\lambda \in Y(G)$ and $P=P_\lambda$.
\end{defn}

If $G$ is connected, then the R-parabolic subgroups
(resp.\ R-Levi subgroups of R-parabolic subgroups)
of $G$ are exactly the parabolic subgroups
(resp.\ Levi subgroups of parabolic subgroups) of $G$;
indeed, most of the theory of parabolic subgroups and
Levi subgroups of connected reductive groups---including rationality properties---carries over to R-parabolic and R-Levi subgroups of
arbitrary reductive groups.
In particular, $R_u(P)(k)$ acts simply transitively on the set of $k$-defined R-Levi subgroups of a $k$-defined R-parabolic
subgroup $P$.  If $P,Q$ are R-parabolic subgroups of $G$ and $P^0= Q^0$,
then $R_u(P)= R_u(Q)$.  Given any maximal torus $T$ of an R-parabolic subgroup $P$, there is a unique R-Levi subgroup $L$ of $P$ such that $T\subseteq L$, and if $P$ is $k$-defined then $L$ is $k$-defined if and only if $T$ is.  If $\lambda$ is $k$-defined then $P_\lambda$ and $L_\lambda$ are $k$-defined.  Conversely, if $G$ is connected and $P$ is a $k$-defined R-parabolic subgroup of $G$ with a $k$-defined Levi subgroup $L$ then there exists $\lambda\in Y_k(G)$ such that $P= P_\lambda$ and $L= L_\lambda$.  (See \cite[Rem.~2.4]{GIT} for a counter-example with $G$ non-connected.)

If $H$ is a subgroup of $G$, then there is an obvious
inclusion $Y(H) \subseteq Y(G)$ of the sets of cocharacters.
When $H$ is reductive and $\lambda \in Y(H)$, there is then an R-parabolic
subgroup of $H$ associated to $\lambda$, as well as an R-parabolic
subgroup of $G$.
In order to distinguish between R-parabolic subgroups associated to
different subgroups of $G$,
we use the notation
$P_\lambda(H)$, $L_\lambda(H)$, etc., where necessary, but we write
$P_\lambda$ for $P_\lambda(G)$ and $L_\lambda$ for  $L_\lambda(G)$.
Note that $P_\lambda(H) = P_\lambda \cap H$,
$L_\lambda(H) = L_\lambda \cap H$
and $R_u(P_\lambda(H)) = R_u(P_\lambda) \cap H$.

More generally, for $H \subseteq G$ a closed subgroup 
which is not necessarily reductive and  
$\lambda \in Y(G)$ a cocharacter normalizing $H$, we define 
$P_\lambda(H) = H \cap P_\lambda$ and 
$R_u(P_\lambda(H)) = H \cap R_u(P_\lambda)$.
We recall a smoothness result from \cite[Prop.\ 2.1.8(3) and Rem.\ 2.1.11]{CGP}
which holds for these intersections.

\begin{prop}
\label{prop:smoothintersection}
Let $H \subseteq G$ be a closed subgroup, and let $\lambda \in Y(G)$ be
a cocharacter that normalizes $H$.
Then the scheme-theoretic intersections
$H \cap P_\lambda$, $H \cap R_u(P_\lambda)$ are smooth and coincide
with $P_\lambda(H)$, $R_u(P_\lambda(H))$ respectively.
\end{prop}

\subsection{$G$-varieties and limits}
\label{subsec:Gvars}

Throughout the paper, $V$ denotes an affine $k$-defined $G$-variety.
This means that we assume both $V$ and the 
action of $G$ on $V$ are $k$-defined.
By a rational $G$-module, we mean a finite-dimensional vector space over $\ovl{k}$ with a linear $G$-action.  If both the $G$-module 
and the action are $k$-defined, then we say the 
rational $G$-module is $k$-defined.  For a subgroup $H$ of $G$, we denote the set of $H$-fixed points in $V$
by $V^H = \{v \in V \mid h\cdot v = v \textrm{ for all } h\in H\}$.

For $v \in V$, let $G_v$ denote the (set-theoretic) stabilizer of $v$,
and let $G_v^0 = (G_v)^0$ denote its identity component.
Let $G(k) \cdot v$ denote the orbit of $v$ under $G(k)$; we call this the {\em rational orbit}.  We write $G\cdot v$ for $G(\ovl{k})\cdot v$ and call this the {\em geometric orbit}.  We say that $G \cdot v$ is \emph{separable} 
provided that the orbit map $G \rightarrow G \cdot v$ is a separable
morphism. Equivalently, $G \cdot v$ is separable if and only if
the scheme-theoretic stabilizer $\mathcal{G}_v$ of $v$ in $G$ is smooth.

For each cocharacter $\lambda \in Y(G)$, we
define a morphism of varieties
$\phi_{v,\lambda}:\ovl{k}^* \to V$ via the formula
$\phi_{v,\lambda}(a) = \lambda(a)\cdot v$.
If this morphism extends to a morphism
$\widehat\phi_{v,\lambda}:\ovl{k} \to V$, then
we say that $\underset{a\to 0}{\lim}\, \lambda(a) \cdot v$ exists,
and set this limit
equal to $\widehat\phi_{v,\lambda}(0)$; note that such an extension,
if it exists, is necessarily unique.  
We sometimes call $v'$ the {\em limit of $v$ along $\lambda$}.  
If $X$ is a closed $G$-stable subset of $V$ and $v\in X$ 
then $\underset{a\to 0}{\lim}\, \lambda(a) \cdot v$ belongs to $X$.

For $\lambda \in Y_k(G)$ we say that
\emph{$\lambda$ destabilizes $v$ over $k$ (for $G$)}
provided
$\underset{a\to 0}{\lim}\, \lambda(a) \cdot v$ exists,
and if
$\underset{a\to 0}{\lim}\, \lambda(a) \cdot v$ exists
and does not belong to $G(k) \cdot v$, then we say
\emph{$\lambda$ properly destabilizes $v$ over $k$  (for $G$)}.

\begin{rem}
\label{rem:linear}
Sometimes
we want to reduce the case of a general ($k$-defined) affine $G$-variety $V$
to the case of a ($k$-defined)
rational $G$-module $V_0$.
Such a reduction is possible, thanks to \cite[Lem.~1.1(a)]{kempf},
for example: given $V$, there is a $k$-defined $G$-equivariant embedding of $V$ inside some $V_0$.
We set up some
standard notation which is in force throughout the paper.

Let $V$ be a rational $G$-module.
Given $\lambda \in Y(G)$ and $n\in\mathbb{Z}$, we define
\begin{align*}
\label{eq:defvlambda0}
V_{\lambda,n} &:=
\{v\in V\mid \lambda(a)\cdot v=a^nv\text{\ for all\ }a\in \ovl{k}^*\},\\
V_{\lambda, \ge0}&:=\sum_{n\ge0} V_{\lambda,n}\ ,
\quad
V_{\lambda, >0}:=\sum_{n>0} V_{\lambda,n}
\quad
\textrm{ and }
\quad
V_{\lambda, <0}:=\sum_{n<0} V_{\lambda,n}.\notag
\end{align*}
Then $V_{\lambda,\ge0}$ consists of the vectors
$v\in V$ such that $\underset{a \ra 0}{\lim} \lambda(a)\cdot v$ exists,
$V_{\lambda,>0}$ is the subset of vectors $v \in V$ such that
$\underset{a \ra 0}{\lim} \lambda(a)\cdot v = 0$,
and $V_{\lambda,0}$ is the subset of vectors $v \in V$ such that
$\underset{a \ra 0}{\lim} \lambda(a)\cdot v = v$.
Furthermore, the limit map
$v\mapsto\underset{a \ra 0}{\lim} \lambda(a)\cdot v$ is nothing
but the projection of $V_{\lambda,\ge0}$ with kernel $V_{\lambda, >0}$
and image $V_{\lambda,0}$.
If the $G$-module $V$ is $k$-defined,
then each $V_{\lambda,n}$ and $V_{\lambda,>0}$,
etc., is $k$-defined (cf.~\cite[II.5.2]{borel}).
\end{rem}

It is sometimes possible to pass from a geometric point in $V$
to an element of $V^n(k)$, using
the following technical lemma.

\begin{lem}
\label{lem:kpoint}
Suppose $V$ is a $k$-defined rational $G$-module.
Let $v \in V$ and let $k_1/k$ be a finite field extension such
that $v \in V(k_1)$.
Let $\alpha_1,\dots,\alpha_n \in k_1$ be a basis for $k_1$ over $k$.
Write $v = \sum_i \alpha_i v_i$ for certain (unique) $v_i \in V(k)$,
and set $\tuple{v} = (v_1,\dots,v_n) \in V^n(k)$.
Let $G$ act diagonally on $V^n$. Then the following assertions hold:
\begin{itemize}
\item[(i)] For $\lambda \in Y_k(G)$, the limit
$\lim_{a\to 0} \lambda(a) \cdot v$ exists if and only if
the limit $\lim_{a\to 0} \lambda(a) \cdot \tuple{v}$ exists.
\item[(ii)] Let $\lambda \in Y_k(G)$ and suppose that the limits
$v' = \lim_{a\to 0} \lambda(a) \cdot v$ and
$\tuple{v}' = \lim_{a\to 0} \lambda(a) \cdot \tuple{v}$ exist.
Then for any $g \in G(k)$, we have 
$v' = g\cdot v$ if and only if $\tuple{v}' = g \cdot \tuple{v}$.
\item[(iii)] We have
$G_{\tuple{v}} \subseteq G_v$ and $G_{\tuple{v}}(k) = G_v(k)$.
\end{itemize}

\end{lem}

\begin{proof}
Parts (i) and (ii) are
\cite[Lem.\ 2.16]{GIT}, while part (iii) follows directly from the definitions.
\end{proof}

The following result asserts that when considering fixed points one may
approximate $k$-split tori by $k$-defined cocharacters.

\begin{lem}
\label{lem:torustocochar}
Given a $k$-split torus $S$ of $G$,
there exists $\mu \in Y_k(S)$ such that $C_G(S) = L_\mu$ and
$V^S = V^{\IM(\mu)}$.
\end{lem}

\begin{proof}
Let $W$ be a $k$-defined finite-dimensional rational $G$-module.
Since $S$ is $k$-split, we can write $W$ as a direct sum of $S$-weight spaces
$W = W_0\oplus\cdots\oplus W_r$ for some $r$,
with the corresponding weights $\alpha_0,\ldots,\alpha_r$ of $S$,
where $\alpha_0$ is the trivial weight.
Given $i>0$, set
$C_i:=\{\sigma\in Y_k(S)\otimes\QQ \mid \langle\sigma,\alpha_i\rangle=0\}$
(where we extend the pairing in the obvious way).
Since $\alpha_i$ is non-trivial for each $i>0$, the subspace $C_i$
is proper for each $i$.
Hence the complement in $Y_k(S) \otimes \QQ$ of $C_1\cup...\cup C_r$
is non-empty.
Taking a point in this complement and multiplying by a suitably
large integer, we can find a cocharacter
$\mu \in Y_k(S)$ such that $\langle \mu,\alpha_i\rangle \neq 0$ for all $i>0$.
Then $W^S = W^{\IM(\mu)}$.

Now consider the affine $G$-variety $X:=G\times V$, where $G$ acts on the first factor by conjugation and on the second factor by
the given action of $G$ on $V$.
By embedding $X$ $G$-equivariantly in a rational $G$-module $W$ and applying the argument in the first paragraph, we can find
a cocharacter $\mu \in Y_k(S)$ such that $X^S = X^{\IM(\mu)}$.
But $X^S = C_G(S)\times V^S$ and $X^{\IM(\mu)} = L_\mu \times V^{\IM(\mu)}$, which gives the result.
\end{proof}

Next we recall a transitivity result for limits along commuting cocharacters.

\begin{lem}
\label{lem:limitoflimit}
Let $v \in V$.
Suppose $\lambda, \mu \in Y(G)$ are commuting cocharacters and
that $v':=\lim_{a\to 0} \lambda(a)\cdot v$ and $v'':=\lim_{a\to0}\mu(a)\cdot v'$ both exist.  Then $\lim_{a\to 0} (n\lambda + \mu)(a)\cdot v = v''$ for all sufficiently large $n$.
\end{lem}

\begin{proof}
It suffices to prove this when $V$ is a vector space, and then the result is contained in \cite[Lem.\ 2.15]{GIT}.
\end{proof}

The next result is similar to Lemma \ref{lem:limitoflimit},
but the assumption is that the limits of $v$ exist along both $\lambda$
and $\mu$.

\begin{lem}
\label{lem:twoways}
Let $v \in V$.
Suppose $\lambda, \mu \in Y(G)$ are commuting cocharacters and
suppose $v \in V$ is such that
$v':=\lim_{a\to 0} \lambda(a)\cdot v$ and $v'':=\lim_{a\to0}\mu(a)\cdot v$ both exist.  Let $m,n\in \NN$.
Then $\lim_{a\to 0} \lambda(a)\cdot v''$, $\lim_{a\to 0} \mu(a)\cdot v'$ and $\lim_{a\to 0} (m\lambda+ n\mu)(a)\cdot v$ all exist and these limits are equal.
\end{lem}

\begin{proof}
It suffices to prove this when $V$ is a vector space.
Then, using the notation from Remark \ref{rem:linear},
we have $v \in V_{\lambda,\geq0}$,
the set of all points for which $\lambda$ acts with a non-negative weight.
This set is closed in $V$ and, since $\lambda$ and $\mu$ commute, it is $\mu$-stable.
Therefore $v'':=\lim_{a\to0}\mu(a)\cdot v\in V_{\lambda,\geq0}$, and hence the limit along $\lambda$ for $v''$ exists.
Similarly, $v' \in V_{\mu,\geq0}$, and hence the limit along $\mu$ for $v'$ exists.  Clearly $v\in U:=V_{\lambda,\geq0}\cap V_{\mu,\geq0}$, so the limit along $m\mu+ n\mu$ for $v$ also exists.

Restricting attention to the subspace $U$, taking the limit along $\lambda$ is the same as projecting onto $U_{\lambda,0} := U \cap V_{\lambda,0}$,
and taking the limit along $\mu$ is the same as projecting onto $U_{\mu,0}:=U\cap V_{\mu,0}$.
The combination of taking the limit along $\lambda$ and then along $\mu$ is the same as taking the limit along $\mu$ and then along $\lambda$,
and the result is simply the projection of $U$ onto $U_{\lambda,0}\cap U_{\mu,0}$.  It is clear that the same is true for taking the limit along $m\mu+ n\mu$ for $v$.
\end{proof}

\begin{lem}
\label{lem:max_tor_ascent}
Let $v \in V$ and 
let $S$ be a maximal $k$-defined torus of $G_v$.
Suppose $\lambda\in Y_k(C_G(S))$ destabilizes $v$
and does not fix $v$.  Then $\lambda$ properly
destabilizes $v$ over $k$ for $G$.
\end{lem}

\begin{proof}
Let $v':= \lim_{a\to 0} \lambda(a)\cdot v$.
Then $S$ and $\IM(\lambda)$ generate
a $k$-defined torus $S'$ of $G_{v'}$.
Now $\IM(\lambda)$ is not contained
in $S$ because $\lambda$ does not fix $v$, so $\dim S'> \dim S$.
But $S$ is a maximal $k$-defined torus of $G_v$, so $v$ and $v'$
cannot be $G(k)$-conjugate.
\end{proof}

\subsection{Results about $R_u(P_\lambda)$-conjugacy}

First we recall two of the main results from \cite{GIT}:

\begin{thm}[{\cite[Thm.\ 3.1]{GIT}}]
\label{thm:rupgalois}
 Let $v\in V$ such that $G_v(k_s)$ 
is $\Gamma$-stable and let $\lambda \in Y_k(G)$ such that
$v':=\lim_{a\to0} \lambda(a)\cdot v$ 
exists and is $R_u(P_\lambda)(k_s)$-conjugate to $v$.
Then $v'$ is $R_u(P_\lambda)(k)$-conjugate to $v$.
\end{thm}

\begin{thm}[{\cite[Thm.\ 3.3]{GIT}}] 
\label{thm:Ruconj}
Let $v\in V$ and suppose $k$ is perfect.
Let $\lambda\in Y_k(G)$ such
that $v':=\underset{a\to 0}{\lim} \lambda(a)\cdot v$
exists and is $G(k)$-conjugate to $v$.
Then $v'$ is $R_u(P_\lambda)(k)$-conjugate to $v$.
\end{thm}

Theorem \ref{thm:Ruconj} was first proved by H.~Kraft and J.~Kuttler
for $k$ algebraically closed of characteristic zero
in case $V = G/H$ is an affine homogeneous space, cf.\
\cite[Prop.\ 2.4]{schmitt} or \cite[Prop.\ 2.1.2]{Gomez}.

The following result is an analogue of Theorem \ref{thm:rupgalois} involving descent to Levi subgroups; the proof involves very similar ideas.

\begin{prop}
\label{prop:ruplevi}
Let $v\in V$.
Suppose $S$ is a $k$-defined torus of $G_v$, $L=C_G(S)$ and $\lambda \in Y_k(L)$ is such that
$v':=\lim_{a\to0} \lambda(a)\cdot v$ exists and is $R_u(P_\lambda)(k)$-conjugate to $v$.
Then $v'$ is $R_u(P_\lambda(L))(k)$-conjugate to $v$.
\end{prop}

\begin{proof}
Set $P = P_\lambda$ and
let $u \in R_u(P)(k)$ be such that $v'=u\cdot v$. Then $u\inverse\cdot\lambda$
fixes $v$.
Let $H$ be the subgroup of $G_v$ generated by the image of $u\inverse\cdot\lambda$ and $S$;
then $H$ is a $k$-defined subgroup of $G_v$, and hence contains a $k$-defined maximal torus
$S'$ with $S \subseteq S'$.
Moreover, since $\lambda$ commutes with $S$ and $P_{u\inverse\cdot\lambda}=P$,
we have $H \subseteq P \cap G_v$.
Since $u\inverse\cdot \lambda$ is a $k$-defined cocharacter of $H$,
we can find $h \in H(k)$ such that $\mu:=hu\inverse\cdot\lambda \in Y_k(S')$.
Since $S\subseteq S'$, we have $\mu \in Y_k(L)$.
Now $\lambda,\mu \in Y_k(L)$ give rise to the same parabolic subgroup $P=P_\lambda=P_\mu$
of $G$, and hence $P_\lambda(L)=P_\mu(L)$.
The R-Levi subgroups $L_\lambda(L)$ and $L_\mu(L)$ are conjugate by an element
$u_0 \in R_u(P_\lambda(L))(k)$,
i.e., $u_0L_\mu(L)u_0\inverse = L_\lambda(L)$.
We claim that $u_0\cdot\mu  = \lambda$.
To see this, note that since $u_0 \in P$, $u_0L_\mu u_0\inverse$
is an R-Levi subgroup of $P$, and this Levi subgroup contains the image of $\lambda$.
Since $R_u(P)$ acts simply transitively on the set of R-Levi subgroups of $P$,
there is only one R-Levi subgroup of $P$ containing the image of $\lambda$, namely $L_\lambda$ itself,
and hence $u_0L_\mu u_0\inverse = L_{u_0\cdot \mu} = L_\lambda$.
But now, since $u_0$, $h$ and $u$ are all elements of $P$,
we have $p:=u_0hu\inverse \in P$ and $u_0\cdot\mu =  p\cdot \lambda$.
Writing $p = u_1l$, with $u_1\in R_u(P)$ and $l\in L_\lambda$,  we have $p\cdot\lambda = u_1\cdot\lambda$ and
$$
L_\lambda = L_{p\cdot\lambda} = L_{u_1\cdot\lambda} = u_1L_{\lambda}u_1\inverse.
$$
Appealing to the simple transitivity of the action of $R_u(P)$ on the R-Levi subgroups of $P$ again,
we see that $u_1=1$, so $u_0\cdot\mu = p\cdot\lambda = \lambda$, as required.
Finally, we have found $u_0\in R_u(P_\lambda(L))(k)$ such that $v':=\lim_{a\to0} \lambda(a)\cdot v$ exists
and $\mu = u_0\inverse\cdot\lambda$ fixes $v$, so by \cite[Lem.\ 2.12]{GIT}, we have $v'=u_0\cdot v$,
which completes the proof.
\end{proof}

\subsection{$k$-rank}

Let $H$ be a subgroup of $G$, not necessarily $k$-defined.  It makes sense to speak of $k$-defined (or $k$-split) subgroups of $H$: a subgroup of $H$ is $k$-defined ($k$-split) if it is $k$-defined ($k$-split) as a subgroup of the $k$-defined group $G$.  Likewise we can speak of $k$-defined cocharacters of $H$.  The notion of maximal $k$-defined (or $k$-split) torus of $H$ has the usual meaning.

Below we will need the following result for $H$ of the form $G_v$ for some $v\in V$.  Note that even when $v\in V(k)$, $G_v$ need not be $k$-defined.

\begin{lem}
\label{lem:k-rank}
 Let $H$ be a subgroup of $G$.  Then any two maximal $k$-split tori of $H$ are $H(k)$-conjugate.
\end{lem}

\begin{proof}
 Let $\widetilde{H}$ be the closure of $\bigcap_{\gamma\in \Gamma} \gamma\cdot H(k_s)$.  
Then $\widetilde{H}$ is $k$-defined by the Galois criterion.  If $S$ is a $k$-defined torus of 
$H$ then $S(k_s)\subseteq \widetilde{H}$, so $S$ is a subgroup of $\widetilde{H}$ as $S(k_s)$ 
is dense in $S$.  Standard results for $k$-defined groups \cite[Thm.~15.2.6]{springer} imply that 
any two maximal $k$-split tori of $\widetilde{H}$ are $\widetilde{H}(k)$-conjugate.  
The assertion of the lemma now follows.
\end{proof}
 
\begin{defn}
 We define the {\em $k$-rank} of $H$ to be the dimension of a maximal $k$-split torus of $H$.
\end{defn}

\section{Cocharacter-closure}
\label{sec:cocharclosure}

Recall the definitions of a cocharacter-closed $G(k)$-orbit and 
of a cocharacter-closed subset of $V$ from the Introduction,
Definitions \ref{defn:cocharclosedorbit} and \ref{defn:cocharclosure}.
We begin with a few elementary observations.

\begin{rems}
(i). Note that the notions in Definition \ref{defn:cocharclosure}
depend on the given action of $G$ on $V$, because
they are made with reference to the limits along cocharacters of $G$.

(ii). A closed $G$-stable set is cocharacter-closed.

(iii). The cocharacter-closed subsets of $V$ form the closed sets of a topology on $V$
(it is clear that arbitrary intersections and unions of cocharacter-closed sets are cocharacter-closed,
and that the empty set and the whole space $V$ are cocharacter-closed).

(iv). Clearly, if a set is cocharacter-closed over $k$, it is cocharacter-closed over $k_0$ where $k_0$ is any subfield of $k$ (such that $V$ is a $k_0$-defined $G$-variety).

(v) Let $X\subseteq V$.  If $X$ is not $G(k)$-stable then $\ccc{X}$ need not be $G(k)$-stable.  For example, suppose $G$ acts freely on $V$ and let $X= \{v\}$ for some $v\in V$.  Then $v$ is not destabilized by any non-zero $\lambda\in Y(G)$: for if $v'= \lim_{a\to 0} \lambda(a)\cdot v$ then $\lambda$ fixes $v'$, which forces $\lambda= 0$ by the freeness of the action.  Hence $\ccc{X}= X= \{v\}$.
\end{rems}

Recall the definition of the preorder $\prec$ from Definition \ref{defn:preorder}.
Examples \ref{ex:g2}
below compares this preorder with the usual 
degeneration order on orbits in the case $k = \ovl{k}$.
See also Example \ref{ex:rsquares}. 
We have the following related notion.

\begin{defn}
\label{def:accessible}
Suppose $v,v' \in V$.
We say that the orbit $G(k) \cdot v'$ is \emph{$1$-accessible from $G(k) \cdot v$}
if there exists $\lambda \in Y_k(G)$ such that $\lim_{a\to 0} \lambda(a)\cdot v$ exists and lies in $G(k) \cdot v'$.
Similarly, for $n\geq 1$, we say that $G(k) \cdot v'$ is \emph{$n$-accessible} from
$G(k) \cdot v$ provided
there exists a finite sequence of $G(k)$-orbits
$G(k) \cdot v=G(k) \cdot v_1, G(k) \cdot v_2,\ldots,G(k) \cdot v_{n+1} = G(k) \cdot v'$
with $G(k) \cdot v_{i+1}$ $1$-accessible from $G(k) \cdot v_i$ for each $1\leq i \leq n$.
We say $G(k)\cdot v'$ is \emph{accessible} from $G(k)\cdot v$ if it is
$n$-accessible for some $n \geq 1$.
\end{defn}

Note that this definition does not depend on the chosen representative of $G(k) \cdot v$ since if
$\lim_{a\to 0} \lambda(a)\cdot v = v''\in G(k) \cdot v'$ and $g \in G(k)$ then
$\lim_{a\to 0} (g\cdot\lambda)(a)\cdot(g\cdot v) = g\cdot v'' \in G(k) \cdot v'$.
It is clear from the definitions that if $G(k) \cdot v'$ is $1$-accessible from $G(k) \cdot v$ then $G(k) \cdot v'\prec G(k) \cdot v$.
Example
\ref{ex:fromf4} below shows that the converse to this is not true, but we do have the following:

\begin{lem}
\label{lem:transitiveclosure}
Suppose $v \in V$.
\begin{itemize}
\item[(i)] We have $\ccc{G(k) \cdot v} = \bigcup G(k) \cdot v'$, where the union is taken over all $v' \in V$ such that
$G(k)\cdot v'$ is accessible from $G(k)\cdot v$.
\item[(ii)] The preorder $\prec$ coincides with the accessibility relation;
that is, given $v,v' \in V$, $G(k) \cdot v'\prec G(k) \cdot v$ if and only if there exists a finite sequence from $G(k) \cdot v$ to $G(k) \cdot v'$ as in Definition \ref{def:accessible}.
\end{itemize}
\end{lem}

\begin{proof}
(i). Let $X$ denote the union defined above. Then given $v' \in X$,
$G(k) \cdot v'$ is accessible from $G(k)\cdot v$.
But now if $\lim_{a \to 0} \lambda(a) \cdot v' = v''$ exists for some $\lambda \in Y_k(G)$, then $G(k) \cdot v''$
is $1$-accessible from $G(k) \cdot v'$, and hence $G(k) \cdot v''$ is accessible from $G(k) \cdot v$, so $v'' \in X$.
Hence $X$ is cocharacter-closed over $k$.
Since $G(k) \cdot v \subseteq X$, we have $\ccc{G(k) \cdot v} \subseteq X$.
But the reverse inclusion is clear: by definition, the cocharacter-closure of $G(k) \cdot v$
must contain all orbits $1$-accessible from $G(k) \cdot v$,
and all orbits $1$-accessible from those orbits, and so on.

(ii). This follows from (i).
\end{proof}

The following elementary example illustrates 
some of the complexities that can arise, 
even over a field of characteristic 0.

\begin{ex}
\label{ex:rsquares}
Let $k=\RR$ and consider the group $G = \GG_m$ acting on $V=\AA^1$ by
$$
a\cdot z := a^2z.
$$
The group $G(k) = \GG_m(k)$ is just the multiplicative group of the field $\RR$, and
there are three orbits of $G(k)$ on $k$-points of $V$: $G(k) \cdot (-1) = \{x \in \RR \mid x<0\}$, $G(k) \cdot 0=\{0\}$ and $G(k) \cdot 1 = \{x\in \RR\mid x>0\}$.
We have $\ccc{G(k) \cdot (-1)} = G(k) \cdot (-1) \cup \{0\}$ and $\ccc{G(k) \cdot 1} = G(k) \cdot 1 \cup \{0\}$.
On the other hand, since the non-zero $G(k)$-orbits $G(k) \cdot 1$ and $G(k) \cdot (-1)$ are both infinite subsets of $V$,
their Zariski closures are the whole of $\AA^1$.
We also have $G\cdot 1 = G\cdot(-1) = \{z \in \AA^1\mid z\neq 0\}$.
This gives an example of how the cocharacter-closure isn't the same as the closure (or the closure intersected with the set of $k$-points) and
how different parts of the same $G$-orbit may be inaccessible from each other when viewed as $G(k)$-orbits.
\end{ex}

For more examples, see Sections \ref{sec:GLn} and \ref{sec:ex}.

\section{The rational Hilbert-Mumford Theorem}
\label{sec:rationalHMT}

In this section we prove Theorem~\ref{thm:rat_HMT}.  We start with a key technical result.

\begin{prop}
\label{prop:technical}
Let $v \in V$.
Suppose $\lambda \in Y_k(G)$ is such that $v':=\lim_{a\to 0}\lambda(a)\cdot v$ exists but is not $R_u(P_\lambda)(k)$-conjugate to $v$.
Let $S$ be a $k$-split torus of $G_v$.
Then there exists $\tau\in Y_k(C_G(S))$ such that $\tilde v:= \lim_{a\to 0}\tau(a)\cdot v$ exists and lies outside $G(k)\cdot v$.
Moreover, there exists a $k$-split torus $\tilde S$ in $G_{\tilde v}$ such that $\dim \tilde S > \dim S$. 
\end{prop}

\begin{proof}
It does no harm to assume that $S$ is a maximal $k$-split torus of $G_v$:
else $S \subset S'$ for a maximal $k$-split torus $S'$ of $G_v$ and
then $Y_k(C_G(S')) \subseteq Y_k(C_G(S))$, so the result for $S$ follows from the result for $S'$.
Using Lemma \ref{lem:torustocochar}, we can find $\mu \in Y_k(S)$ such that $V^S = V^{\IM(\mu)}$ and $L:=C_G(S) = L_\mu$.
Now let $T$ be a maximally split $k$-defined maximal torus of $P_\lambda \cap P_\mu$. 
There exists $u \in R_u(P_\lambda)(k)$ such that $u\cdot\lambda \in Y_k(T)$.
Moreover, $\lim_{a\to 0} (u\cdot\lambda)(a)\cdot v = u\cdot v'$ exists and cannot be $R_u(P_\lambda)(k)$-conjugate to $v$, since
then $v'$ would also be $R_u(P_\lambda)(k)$-conjugate to $v$, contradicting our hypothesis.
Hence we may replace $\lambda$ with $u\cdot\lambda$ and $v'$ with $u\cdot v'$ if we like and assume that $\lambda \in Y_k(T)$.

Now there exists $u_1 \in R_u(P_\mu)(k)$ such that $\nu: = u_1\cdot\mu$ belongs to $Y_k(T)$, and since $\mu$ fixes $v$, $\lim_{a \to 0} \nu(a)\cdot v = u_1\cdot v$ exists.
Since $\lambda$ and $\nu$ belong to $Y_k(T)$, they commute.
By Lemma~\ref{lem:twoways}, $v'':=\lim_{a\to 0} (\lambda + n\nu)(a)\cdot v$ exists for all positive integers $n$, and this limit does not depend on $n$.
Choosing $n$ sufficiently large, we may assume that $P_{\lambda + n\nu} \subseteq P_\nu = P_\mu$ and $R_u(P_\mu) \subseteq R_u(P_{\lambda+n\nu})$ \cite[Lem.~2.15]{GIT}.
In particular, $u_1 \in R_u(P_{\lambda + n\nu})(k)$.
Let $\sigma = \lambda + n\nu$ for such a choice of $n$.
Note that, since $v''$ can be obtained as a limit along $\lambda$ and as a limit along $\nu$ (Lemma~\ref{lem:twoways}), $\lambda$ and $\nu$ both fix $v''$.
Now $\nu = u_1\cdot\mu$; so $u_1Su_1\inverse \subseteq G_{v''}$, since $V^S = V^{\IM(\mu)}$, and $\IM(\lambda)$ commutes
with $u_1Su_1\inverse$, since $L_\mu = C_G(S)$.
Hence $\IM(\lambda)$ and $u_1Su_1\inverse$ generate a 
$k$-defined torus $S''$ of $G_{v''}$ of dimension at least as large as the dimension of $S$.
Note that $S''$ is $k$-split, by \cite[Thm.\ 15.4 (i)]{borel}.

If $S'' = u_1Su_1\inverse$, then $\IM(\lambda) \subseteq u_1Su_1\inverse$,
and therefore $\lambda$ fixes $u_1\cdot v$.
Since $\lambda$ evaluates in $P_\mu$, $u_1 \in R_u(P_\mu)(k)$ and $R_u(P_\mu)$ is a closed connected normal subgroup of $P_\mu$,
we can write $u_1 = xu_2$ with $x \in P_{-\lambda}(k)$, $u_2 \in R_u(P_\lambda)(k)$ by \cite[Thm.\ 13.4.2, Cor.\ 13.4.4]{springer}.
But now \cite[Lem.\ 2.13]{GIT} implies that $v' = u_2\cdot v \in R_u(P_\lambda)(k)\cdot v$, which contradicts the hypothesis on $\lambda$.
Hence we must conclude that $\dim S'' > \dim S$. In particular, $v$ and $v''$ are not $G(k)$-conjugate, since $S$ is a maximal $k$-split torus of $G_v$.

Finally, recall that $u_1 \in R_u(P_\sigma)(k)$, so $\tilde v := \lim_{a\to 0} (u_1\inverse\cdot\sigma)(a)\cdot v = u_1\inverse\cdot v''$ exists
and is also not $G(k)$-conjugate to $v$.
But $u_1\inverse\cdot\sigma = u_1\inverse\cdot\lambda + n\mu \in Y_k(L_\mu) = Y_k(C_G(S))$ so, 
setting $\tau=u_1\inverse\cdot \sigma$ and $\tilde S = u_1\inverse S'' u_1$,
we are done.
\end{proof}

We note here a consequence of Proposition \ref{prop:technical} which we 
use at the end of Section \ref{sec:up_and_down}.

\begin{cor}
\label{cor:birkes1}
Let $v \in V$.
If $G_v$ contains a maximal $k$-split torus of $G$, then $G(k) \cdot v$ is cocharacter-closed over $k$.
\end{cor}

\begin{proof}
Let $S$ be a  maximal $k$-split torus of $G$ contained in $G_v$.
The result now follows from Proposition \ref{prop:technical} applied to $S$. 
\end{proof}

We can now give our generalisation of the Hilbert-Mumford Theorem.  Recall in particular the notions of cocharacter-closed and cocharacter-closure, 
Definitions \ref{defn:cocharclosedorbit} and 
\ref{defn:cocharclosure},
the preorder $\prec$ on the $G(k)$-orbits of $V$,
Definition \ref{defn:preorder}, 
and the relationship of these notions with the notion of accessibility,
Lemma \ref{lem:transitiveclosure}.
The key result, which proves Theorem~\ref{thm:rat_HMT}, is the following.

\begin{thm}[The rational Hilbert-Mumford Theorem]
\label{thm:uniqueaccessible}
Suppose $v \in V$.
\begin{itemize}
\item[(i)] There is a unique cocharacter-closed (over $k$) $G(k)$-orbit $1$-accessible
from $G(k) \cdot v$.
\item[(ii)] The orbit from part (i) is the unique
cocharacter-closed $G(k)$-orbit in $\ccc{G(k) \cdot v}$.
\item[(iii)] If $G(k) \cdot v'\prec G(k) \cdot v$, then the unique cocharacter-closed $G(k)$-orbits in $\ccc{G(k) \cdot v}$ and $\ccc{G(k) \cdot v'}$ are equal.
\end{itemize}
\end{thm}

\begin{proof}
(i). We first show existence. If $G(k) \cdot v$ is already cocharacter-closed over $k$, then there is nothing to do.
Otherwise, there exists $\lambda \in Y_k(G)$ such that $v' = \lim_{a\to 0}\lambda(a)\cdot v$ exists and is not $G(k)$-conjugate to $v$.
By Proposition \ref{prop:technical}, we can even assume that $\lambda \in Y_k(C_G(S))$, where $S$ is a maximal $k$-split torus of $G_v$,
and that the  $k$-rank of $G_{v'}$ is strictly greater than the $k$-rank of $G_v$.
If $G(k) \cdot v'$ is cocharacter-closed over $k$ then we are done.
If not, repeat the process to find $\mu \in Y_k(C_G(S_1))$ such that $v''=\lim_{a\to 0}\mu(a)\cdot v'$ exists
and $G_{v''}$ has strictly greater $k$-rank than $G_{v'}$.
Note that, since $\mu \in Y_k(C_G(S_1))$ and $\IM(\lambda) \subseteq S_1$, $\lambda$ and $\mu$ commute,
so by Lemma \ref{lem:limitoflimit} $G(k) \cdot v''$ is $1$-accessible from $G(k) \cdot v$.
Again, if $G(k) \cdot v''$ is cocharacter-closed over $k$ then we are done.
Otherwise, repeat the process again.
Since the $k$-rank of the stabilizer increases at each step, the process must terminate at a cocharacter-closed orbit which is
$1$-accessible from $G(k) \cdot v$.

For uniqueness, suppose $G(k) \cdot v_1$ and $G(k) \cdot v_2$ are two cocharacter-closed orbits $1$-accessible from $G(k) \cdot v$.
 Choose $\lambda_1, \lambda_2 \in Y_k(G)$ such that $\lim_{a\to 0} \lambda_i(a)\cdot v\in G(k)\cdot v_i$ for $i= 1,2$.  Without loss we can assume that $\lim_{a\to 0} \lambda_i(a)\cdot v= v_i$ for $i= 1,2$.
Then $P_{\lambda_1} \cap P_{\lambda_2}$ contains a maximally split $k$-defined maximal torus of $G$ and we can conjugate $\lambda_1$ (resp.\ $\lambda_2$) by an element
of $R_u(P_{\lambda_1})(k)$ (resp.\ an element of $R_u(P_{\lambda_2})(k)$) so that $\lambda_1$ and $\lambda_2$ commute.
Let $v'= \lim_{a\to 0} \lambda_1(a)\cdot v_2= \lim_{a\to 0} \lambda_2(a)\cdot v_1$ (these limits exist and are equal by Lemma~\ref{lem:twoways}). Since $G(k) \cdot v_i$ is cocharacter-closed over $k$ for $i=1,2$, $v'$ is $G(k)$-conjugate to $v_1$ and $v_2$,
and hence $G(k) \cdot v_1 = G(k) \cdot v' = G(k) \cdot v_2$, as required. This completes the proof of (i).

(ii). Suppose $G(k) \cdot v'$ is the unique cocharacter-closed orbit $1$-accessible from $G(k) \cdot v$, as provided by part (i),
and suppose $G(k) \cdot v''$ is another cocharacter-closed orbit in $\ccc{G(k) \cdot v}$.
By Lemma \ref{lem:transitiveclosure}(i), there is a finite sequence $G(k) \cdot v = G(k) \cdot v_1,G(k) \cdot v_2,\ldots,G(k) \cdot v_n = G(k) \cdot v''$
of orbits with $G(k) \cdot v_{i+1}$ $1$-accessible from $G(k) \cdot v_i$ for each $1\leq i \leq n-1$.
By choosing our representatives suitably, we have cocharacters $\lambda$ and $\mu \in Y_k(G)$ such that
$\lim_{a\to 0} \lambda(a)\cdot v_{n-2} = v_{n-1}$ and $\lim_{a\to 0} \mu(a)\cdot v_{n-1} = v_n = v''$.
Now, by the usual argument, we can find $u_1 \in R_u(P_\lambda)(k)$ and $u_2 \in R_u(P_\mu)(k)$
such that $\sigma:=u_1\cdot\lambda$ and $\tau:=u_2\cdot\mu$ commute.
Moreover, $\lim_{a\to 0} \sigma(a) \cdot v_{n-1} = u_1\cdot v_{n-1}$ and $\lim_{a\to 0} \tau(a)\cdot v_{n-1} = u_2\cdot v''$.
Hence by Lemma~\ref{lem:twoways},
$$
\lim_{a\to 0} \sigma(a)\cdot (u_2\cdot v'') \textrm{ exists and equals } \lim_{a\to 0} \tau(a)\cdot(u_1\cdot v_{n-1}).
$$
Call this common limit $w$.
Since the orbit $G(k) \cdot v''$ is cocharacter-closed over $k$, we have that $G(k) \cdot w = G(k) \cdot v''$.
But, since $u_1\cdot v_{n-1} = \lim_{a\to0} \sigma(a) \cdot v_{n-2}$ and $\sigma$ and $\tau$ commute,
we can apply Lemma \ref{lem:limitoflimit} to conclude that there exists $n \in \NN$ such that
$\lim_{a\to 0} (n\sigma + \tau)(a)\cdot v_{n-2} = w \in G(k) \cdot v''$; in particular, $G(k) \cdot v''$ is $1$-accessible from $G(k) \cdot v_{n-2}$.
Continuing in this way, we conclude that $G(k) \cdot v''$ is $1$-accessible from $G(k) \cdot v$ and hence $G(k) \cdot v''=G(k) \cdot v'$ by the uniqueness in part (i).
This completes the proof of (ii).

(iii). If $G(k) \cdot v'\prec G(k) \cdot v$, then $v' \in \ccc{G(k) \cdot v}$ and hence $\ccc{G(k) \cdot v'} \subseteq \ccc{G(k) \cdot v}$.
This means that any cocharacter-closed orbit in $\ccc{G(k) \cdot v'}$ is also contained in $\ccc{G(k) \cdot v}$,
and the uniqueness from part (ii) gives the result.
\end{proof}

\begin{rems}
\label{rem:Levy}
(i).
See Example \ref{ex:charpol} for a linear algebra example illustrating
the uniqueness of the cocharacter-closed orbit in the cocharacter-closure.

(ii).
In \cite[Thm.\ 3.4]{levy2}, Levy considers the case where $k$ is perfect and proves
the following result: Let $v \in V(k)$ and assume that two limits
$s_1 = \lim_{a \to 0} \lambda(a) \cdot v$, $s_2 = \lim_{a \to 0} \mu(a) \cdot v$ along
$k$-defined cocharacters $\lambda$ and $\mu$ exist. Further assume that both 
$G \cdot s_1$ and $G \cdot s_2$ are Zariski-closed. Then $s_1$ and $s_2$ are $G(k)$-conjugate.
Since $k$ is assumed to be perfect, it is known that the orbits through
$s_1$ and $s_2$ are Zariski-closed if and only if they are cocharacter-closed
over $k$, compare Theorem \ref{thm:descent+ascent} and Proposition \ref{prop:perfectorseparable}(i) below.
Hence Levy's conjugacy result follows from the uniqueness assertion
of Theorem \ref{thm:uniqueaccessible},
and moreover holds for arbitrary (possibly non-rational) points $v \in V$.
\end{rems}

The following result is a rational version of the form of the Hilbert-Mumford Theorem given in \cite[Thm.\ 1.4]{kempf}.  In our setting, it is a direct consequence
of Theorem \ref{thm:uniqueaccessible}.

\begin{cor}
\label{cor:rationalHMT}
Suppose $X$ is a $G(k)$-stable co\-cha\-racter-closed
subset of $V$ which meets $\ccc{G(k) \cdot v}$.
Then there exists $\lambda\in Y_k(G)$
such that $\lim_{a\to 0} \lambda(a)\cdot v$ exists and lies in $X$.
\end{cor}

\begin{proof}
Since $X \cap \ccc{G(k) \cdot v}$ is $G(k)$-stable, it contains a $G(k)$-orbit.
Since the intersection is cocharacter-closed, it contains the cocharacter-closure of this orbit, and hence it contains
the corresponding unique cocharacter-closed orbit provided by Theorem \ref{thm:uniqueaccessible}(ii).
However, this cocharacter-closed orbit must also be the unique cocharacter-closed orbit contained in $\ccc{G(k) \cdot v}$,
by Theorem \ref{thm:uniqueaccessible}(iii).
Since this orbit is $1$-accessible from $G(k) \cdot v$ by Theorem \ref{thm:uniqueaccessible}(i),
we get the result required.
\end{proof}

\begin{rem}
\label{rem:no_optimality}
 We do not know whether the strengthened version of the Hilbert-Mumford Theorem discussed in the Introduction can be extended to arbitrary fields.  For more on this, see \cite[Sec.~1]{GIT}.
\end{rem}

\section{Ascent and descent}
\label{sec:up_and_down}

Given a field extension $k'/k$,
we use the terminology ``Galois descent''
to refer to a result which guarantees that
a given $G(k)$-orbit is cocharacter-closed over $k$ provided the corresponding
$G(k')$-orbit is cocharacter-closed over $k'$, and vice versa for ``Galois ascent''.
Likewise, we refer to ``Levi descent'' if the result at hand implies
that a given $L(k)$-orbit (for $L$ a $k$-defined Levi subgroup) is cocharacter-closed over $k$ provided the corresponding
$G(k)$-orbit is cocharacter-closed over $k$, and vice versa for ``Levi ascent''.

We begin this section with a further consequence of
Proposition \ref{prop:technical}.

\begin{cor}
\label{cor:mainconjcocharclosed}
Let $v \in V$.
Suppose $G(k)\cdot v$ is cocharacter-closed over $k$.
Then for all $\lambda \in Y_k(G)$ such that $v':=\lim_{a\to0}\lambda(a)\cdot v$ exists,
$v'$ is $R_u(P_\lambda)(k)$-conjugate to $v$.
\end{cor}

\begin{proof}
Suppose, for contradiction, that we have $\lambda \in Y_k(G)$ such that $v'$ exists but is not $R_u(P_\lambda)(k)$-conjugate to $v$.
Then, by Proposition \ref{prop:technical}, there exists a point $v''$ which can be obtained as the limit of $v$ along a
$k$-defined cocharacter of $G$ but is not $G(k)$-conjugate to $v$.
But this contradicts the assumption that $G(k)\cdot v$ is cocharacter-closed over $k$.
\end{proof}

\begin{rem}
Corollary \ref{cor:mainconjcocharclosed}
generalizes \cite[Thm.\ 3.10]{GIT}
by removing the connectedness assumption on $G$.
\end{rem}

As a direct consequence of Corollary \ref{cor:mainconjcocharclosed}, we see that cocharacter-closedness only depends on the identity component of $G$.

\begin{cor}
\label{cor:conn_nonconn}
Let $v \in V$.
The orbit
$G(k)\cdot v$ is cocharacter-closed over $k$ if and only if $G^0(k)\cdot v$ is cocharacter-closed over $k$.
\end{cor}

\begin{proof}
It is immediate from the definition that if $G^0(k)\cdot v$
is cocharacter-closed over $k$,
then $G(k)\cdot v$ is cocharacter-closed over $k$.
Conversely, suppose $G^0(k)\cdot v$ is not cocharacter-closed over $k$.  Then there exists $\lambda\in Y_k(G^0)= Y_k(G)$ such that $\lambda$ destabilizes $v$ and $\lim_{a\to 0}\lambda(a)\cdot v$ does not belong to $R_u(P_\lambda(G^0))(k)\cdot v= R_u(P_\lambda(G))(k)\cdot v$.  It follows from
Corollary \ref{cor:mainconjcocharclosed} that $G(k)\cdot v$ is not cocharacter-closed over $k$, as required.
\end{proof}

Here is our main result on Levi descent and split Levi ascent; it proves Theorem~\ref{thm:Gvkdef}(iii).

\begin{thm}
\label{thm:leviascentdescent}
Let $v \in V$.
Suppose $S$ is a $k$-defined torus of $G_v$ and set $L = C_G(S)$.
\begin{itemize}
\item[(i)] If $G(k)\cdot v$ is cocharacter-closed over $k$, then $L(k)\cdot v$ is cocharacter-closed over $k$.
\item[(ii)] If $S$ is $k$-split, then $G(k)\cdot v$ is cocharacter-closed over $k$ if and only if $L(k)\cdot v$ is cocharacter-closed over $k$.
\end{itemize}
\end{thm}

\begin{proof}
(i). Suppose $G(k)\cdot v$ is cocharacter-closed over $k$, and let $\lambda \in Y_k(L)$
be such that $v':=\lim_{a\to 0}\lambda(a)\cdot v$ exists.
Then, by Corollary \ref{cor:mainconjcocharclosed}, $v'$ is $R_u(P_\lambda)(k)$-conjugate to $v$.
Now apply Proposition \ref{prop:ruplevi} to get that $v'$ is $R_u(P_\lambda(L))(k)$-conjugate---and hence $L(k)$-conjugate---to $v$.

(ii). The forward implication follows from part (i).
Now suppose that $G(k)\cdot v$ is not cocharacter-closed over $k$.
Then, by Proposition \ref{prop:technical}, there exists $\tau \in Y_k(L)$ such that $\lim_{a\to 0}\tau(a)\cdot v$ exists but is not
$G(k)$-conjugate to $v$.
But then this limit cannot be $L(k)$-conjugate to $v$, and hence $L(k)\cdot v$ is not cocharacter-closed over $k$.
\end{proof}

Our next result---which proves the second assertion of Theorem~\ref{thm:Gvkdef}(ii)---is Galois descent for separable algebraic extensions.

\begin{prop}
\label{prop:galois_descent}
Let $v\in V$ such that $G_v(k_s)$ is $\Gamma$-stable 
and let $k'/k$ be a separable algebraic extension.
If $G(k')\cdot v$ is cocharacter-closed over $k'$ then $G(k)\cdot v$ is cocharacter-closed over $k$.
\end{prop}

\begin{proof}
Suppose $G(k')$ is cocharacter-closed over $k'$.  Let $\lambda\in Y_k(G)$ such that $v':= \lim_{a\to 0} \lambda(a)\cdot v$ exists.
By Corollary \ref{cor:mainconjcocharclosed}, $v'$ is $R_u(P_\lambda)(k')$-conjugate to $v$.
In particular, $v'$ is $R_u(P_\lambda)(k_s)$-conjugate to $v$.
The result now follows from Theorem \ref{thm:rupgalois}.
\end{proof}

It turns out to be more subtle to prove converse statements
to Proposition \ref{prop:galois_descent}---i.e., Galois ascent results---under appropriate hypotheses on $v$.  Such a result would follow for $v\in V(k)$ if we could prove a rational version of the strengthened Hilbert-Mumford Theorem (cf.\ Remark~\ref{rem:no_optimality}): for then the optimal $k_s$-defined cocharacter $\lambda_{\rm opt}$ would be $\Gamma$-fixed and hence would be $k$-defined.  As a first lemma, we prove that $G(k)\cdot v$ is
not cocharacter-closed over $k$, provided the limit along
some $k_s$-defined cocharacter lies outside the geometric orbit.

\begin{lem}
\label{lem:geom_galois_ascent}
Let $v\in V(k)$.
Suppose there exists $\lambda\in Y_{k_s}(G)$
such that $\lambda$ properly destabilizes $v$ over ${\bar k}$.
Then $G(k)\cdot v$ is not cocharacter-closed over $k$.
\end{lem}

\begin{proof}
Let $v' = \lim_{a\to 0} \lambda(a) \cdot v \in V(k_s)$ and let
$Y \subseteq V(k_s)$ be the set of Galois-conjugates of $v'$; note that $Y$ is finite.
Set $S = \bigcup_{y\in Y} \overline{G\cdot y}$.
Since $Y$ is finite, $S$ is closed.
This closed set is also $G$-stable, $\Gamma$-stable and $k_s$-defined, hence $k$-defined.
Note also that since $\lambda$ properly destabilises $v$ to $v'$, the dimension of the $G$-orbit of $v'$ is strictly smaller than that of $v$, and hence all $G$-orbits in $S$ have dimension strictly smaller than that of $G\cdot v$.
This implies that $S \cap G\cdot v  = \emptyset$;
in particular, $S \cap G(k)\cdot v = \emptyset$.
Now, using the terminology of \cite[\S4]{GIT}, $v$ is uniformly $S$-unstable over $k_s$, so
we may thus apply \cite[Cor.\ 4.9]{GIT} to conclude that $v$ is uniformly $S$-unstable over $k$ (note that $\{v\}$ is $k$-closed).
In other words, there exists a $k$-defined cocharacter $\mu$ with
$\lim_{a\to 0} \mu(a) \cdot v \in S$.
The claim follows.
\end{proof}

We get stronger ascent and descent results under some 
assumptions of $k$-definability on $G_v$.
In general, the stabilizer $G_v$ is only $k_i$-defined but not $k$-defined.
If $G_v$ is $k$-defined, then we can say more.
In particular, for $v \in V(k)$, the stabilizer $G_v$ is
$k$-defined if $G\cdot v$ is separable or if $k$ is perfect,
see Proposition \ref{prop:perfectorseparable}.  The following result proves Theorem \ref{thm:Gvkdef}(ii) (second assertion) and (iii).

\begin{thm}
\label{thm:descent+ascent}
Let $v\in V$.
Suppose $G_v$ has a maximal torus that is $k$-defined.
Then the following hold.
\begin{itemize}
\item[(i)] Suppose $v\in V(k)$.  For any separable algebraic extension $k'/k$, $G(k')\cdot v$ is cocharacter-closed over $k'$ if and only if $G(k)\cdot v$ is cocharacter-closed over $k$.
\item[(ii)] Let $S$ be a $k$-defined torus of $G_v$ and let $L = C_G(S)$.  Then $L(k)\cdot v$ is cocharacter-closed over $k$ if and only if $G(k)\cdot v$ is cocharacter-closed over $k$.
\end{itemize}
\end{thm}

\begin{proof}
 The forward implication of (i) follows from Proposition \ref{prop:galois_descent}.
We prove the reverse implication of (i) using induction on $\dim G$.  By Remark \ref{rem:linear},
we may assume that $V$ is a $k$-defined rational $G$-module.
The result holds trivially if $\dim G= 0$.
Assume the result holds for any $G'$ such that $\dim G'< \dim G$.
Let $k'/k$ be a separable algebraic extension.  Suppose $G(k')\cdot v$ is not cocharacter-closed over $k'$.  We prove that $G(k)\cdot v$ is not cocharacter-closed over $k$.  By Proposition~\ref{prop:galois_descent}, we can assume that $k'= k_s$.  So some $\lambda\in Y_{k_s}(G)$ properly destabilizes $v$ over $k_s$.  If $\lambda$ properly destabilizes $v$ over $\bar{k}$,
then $G(k)\cdot v$ is not cocharacter-closed over
$k$, by Lemma~\ref{lem:geom_galois_ascent}.

So suppose $\lambda$ does not properly destabilize $v$ over ${\bar k}$.
Then $u\cdot \lambda$ centralizes $v$ for some $u\in R_u(P_\lambda)$, 
by Theorem \ref{thm:Ruconj} for  ${\bar k}$.  By hypothesis, there exists a
$k$-defined maximal torus $S'$ of $G_v$.  Set $G'= C_G(S')$,
a reductive $k$-defined subgroup of $G$.  Since $\lambda$ properly destabilizes $v$ over $k_s$,
$u\cdot\lambda\neq \lambda$, which implies that $\IM (u\cdot \lambda)$
is not contained in $Z(G^0)$.  It follows that $S'\not\subseteq Z(G^0)$, so $\dim G'< \dim G$.
By Theorem \ref{thm:leviascentdescent}(ii)
applied to the $k_s$-split torus $S'$,
$G'(k_s)\cdot v$ is not cocharacter-closed over $k_s$.
By the induction hypothesis, $G'(k)\cdot v$ is not cocharacter-closed over $k$.
So choose $\mu\in Y_k(G')$ such that $v':= \lim_{a\to 0} \mu(a)\cdot v$ exists
and is not $G'(k)$-conjugate to $v$.  Then in particular, $\mu$
does not fix $v$.  Hence $\mu$ properly destabilizes $v$ over $k$
for $G$, by Lemma~\ref{lem:max_tor_ascent}, and we are done.

Finally, we prove part (ii).
We want to apply (i) to both $G$ and $L$, so we first check
that $L_v$ has a maximal torus that is $k$-defined.
By assumption, $G_v$ has a maximal torus $T$ that is $k$-defined.
Let $H= G_v$ and let $\widetilde{H}$ be as in the proof of Lemma~\ref{lem:k-rank}; then $S\subseteq \widetilde{H}$ and $T$ is a maximal torus of $\widetilde{H}$.   As $S$ and $\widetilde{H}$ are $k$-defined, we can choose a $k$-defined maximal torus $T'$ of $\widetilde{H}$ containing $S$.  Then $\dim T'= \dim T$, so $T'$ is a maximal torus of both $G_v$ and $L_v$, as required.

If $v\in V(k)$ then the result follows from Theorem~\ref{thm:leviascentdescent}.  More generally, if $k$ is infinite then we can first apply Lemma~\ref{lem:kpoint} and replace an arbitrary $v$ with some $\tuple{v}\in V(k)^n$ (note that $S\subseteq G_{\tuple{v}}$ as $S(k)$ is dense in $S$).  For arbitrary $k$ and $v$ we need the following argument, which relies on constructions from Section~\ref{sec:k-point}.  We use Theorem~\ref{thm:k-point} below to replace $V$ with another $k$-defined $G$-variety $W$ and $v$ with some $w \in W(k)$.
By Theorem \ref{thm:k-point}(i), $L(k) \cdot v$ is 
cocharacter-closed if and only if $L(k) \cdot w$ is so,
and likewise for $G(k) \cdot v$. Moreover,
Theorem \ref{thm:k-point}(ii) assures that $S \subseteq G_w$
and that $G_w$ has a maximal torus that is $k$-defined.
We may thus apply Theorem~\ref{thm:descent+ascent}(i) and assume $k=k_s$. But then
$S$ is $k$-split, so the result follows from
Theorem \ref{thm:leviascentdescent}.
\end{proof}

\begin{cor}
\label{cor:birkes1_plus}
Let $v \in V$.
If $G_v$ contains a maximal $k$-torus of $G$, then $G(k) \cdot v$ is cocharacter-closed over $k$.
\end{cor}

\begin{proof}
 Suppose $G_v$ contains a maximal $k$-torus $S$ of $G$.  Then $S$ is a $k$-defined maximal torus of $G$, so $S$ is a $k$-defined maximal torus of $G_v$.  By Theorem~\ref{thm:descent+ascent}(i), it is enough to show that $G(k_s) \cdot v$ is cocharacter-closed over $k_s$.  But this follows from Corollary~\ref{cor:birkes1}, as $S$ splits over $k_s$.
\end{proof}

We turn now to the proof of Theorem \ref{thm:anisotropy}(i). 

\begin{proof}[Proof of Theorem \ref{thm:anisotropy}(i)]
Suppose $G$ is $k$-anisotropic. 
Let $W$ be an affine $k$-defined $G$-variety and let $w \in W$.
Since $Y_k(G) = \{0\}$,  the 
$G(k)$-orbit of $w$ in $W$ is cocharacter-closed over $k$.
(Note that we do not require that $w \in W(k)$ here.)
The converse follows from the argument in the proof of 
\cite[Lem.\ 10.1]{Birkes}.
\end{proof}

\begin{rems}
\label{rem:anisotropy}
(i). We recover the result (cf.~\cite[Cor.\ 3.8]{boreltits})
that if $k$ is perfect and $G$ is $k$-anisotropic,
then every element in $G(k)$ is semisimple.
 For if $g \in G(k)$ is not semisimple, 
then its $G(\ovl{k})$-conjugacy class in $G$ is not 
closed, contradicting Theorem \ref{thm:anisotropy}(ii).

(ii). 
Theorem \ref{thm:anisotropy}(ii) fails for 
non-perfect fields.
In \cite[p.\ 488]{GilleQueguiner-Mathieu}, Gille and Qu\'eguiner-Mathieu
give an example of a $k$-anisotropic semisimple 
group $G$ of the form $G = \PGL_1(A)$, 
where $A$ is 
a simple central division algebra over $k$
containing a field $K$ such that $K/k$ is purely inseparable.
Moreover, $G$ contains a smooth unipotent subgroup
admitting non-trivial $k$-points.
If $1 \neq g \in G(k)$ is unipotent, then its $G(\ovl{k})$-conjugacy class
is not closed. So the conclusion of Theorem \ref{thm:anisotropy}(ii) does not
hold in this instance.
\end{rems}

We finish with the following result, which is an easy consequence of 
Corollary \ref{cor:birkes1} and Theorem \ref{thm:Gvkdef}(ii);
it shows that \cite[Property (B)]{Birkes} holds for any perfect field
(cf.~\cite[Thm.\ 7.2, Thm.\ 8.2]{Birkes}).

\begin{prop}
\label{prop:birkes}
Suppose $k$ is perfect. Let $v \in V$.
If $G_v$ contains a maximal $k$-split torus of $G$, then $G \cdot v$ is closed.
\end{prop}

\begin{rem}
 Proposition~\ref{prop:birkes} fails for non-perfect fields: Remark~\ref{rem:anisotropy}(ii) gives an example of this with $G$ anisotropic.  The assertion of Proposition~\ref{prop:birkes} does hold for any $k$, however, if we assume that $G$ is split: for if $G_v$ contains a maximal $k$-split torus $S$ of $G$ then $S$ is a maximal torus of $G$, so $G\cdot v$ is closed by Theorem~\ref{thm:descent+ascent}(ii) applied to the field $\ovl{k}$.
\end{rem}

\section{Reduction to $k$-points}
\label{sec:k-point}

In this section we prove the following result which makes it possible
to pass from a geometric point in $V$ to a rational point in another
affine variety $W$.

\begin{thm}
\label{thm:k-point}
Let $v \in V$. Then there exists an affine $G$-variety $W$ and $w \in W(k)$
with the following properties:
\begin{itemize}
\item[(i)] For any reductive $k$-defined subgroup $M$ of $G$,
$M(k) \cdot v$ is cocharacter-closed over $k$ if and only
if $M(k) \cdot w$ is cocharacter-closed over $k$.
\item[(ii)] Let $H \subseteq G$ be a connected $k$-defined subgroup.
Then $H \subseteq G_v$ if and only if $H \subseteq G_w$.
\end{itemize}
\end{thm}

To prove Theorem \ref{thm:k-point} we require a series of lemmas.
We first reduce from geometric points to $k_s$-points. 

\begin{lem}
\label{lem:k_s-point}
 To prove Theorem \ref{thm:k-point}, we may assume $v \in V(k_s)$.
\end{lem}

\begin{proof}
By Remark \ref{rem:linear}, we may assume that $V$ is a $k$-defined
rational $G$-module. 
Let $k_1/k_s$ be a finite field extension such that
$v \in V(k_1)$. We now apply Lemma \ref{lem:kpoint} to produce
$\tuple{v} \in V^n(k_s)$ for suitable $n$.
It follows from Lemma \ref{lem:kpoint}(i) and (ii) (for the field $k_s$) that
$M(k) \cdot v$ is cocharacter-closed if and only if 
$M(k) \cdot \tuple{v}$ is so. 

Suppose $H$ is a connected $k$-defined subgroup
of $G$.
If $H$ is contained in $G_\tuple{v}$, then 
by Lemma \ref{lem:kpoint}(iii)
we have $H \subseteq G_\tuple{v} \subseteq G_v$.
Conversely, suppose that $H \subseteq G_v$.
Since $H(k_s)$ is dense in $H$, it is enough to show
$H(k_s) \subseteq G_\tuple{v}$, which follows from the
equality $G_\tuple{v}(k_s) = G_v(k_s)$.
We conclude that we may replace $V$ by $V^n$ and
$v$ by $\tuple{v}$ to prove Theorem \ref{thm:k-point}.
\end{proof}

The next step is to make $v \in V(k_s)$ become stable under the action
of the Galois group. We may achieve this by passing from $V$ to the quotient
by a finite group, using the following result.

\begin{lem}
\label{lem:finquot}
Let $F$ be a finite abstract group, regarded as a $k$-defined algebraic group
endowed with the usual $k$-structure (so that $F = F(k)$).
Suppose $F\times G$ has a $k$-defined action on $V$. 
Then the action of $G$ on $V$ descends to give a $k$-defined action of $G$ on $V/F$.
Let $\pi\colon V\ra V/F$ be the canonical $G$-equivariant projection.
Then the following hold for $v \in V$:
\begin{itemize}
\item[(i)]
Let $\lambda\in Y(G)$.
Then $\lambda$ destabilizes $v$ if and only if $\lambda$ destabilizes $\pi(v)$.
\item[(ii)]
For any reductive $k$-defined subgroup $M \subseteq G$,
$M(k)\cdot v$ is cocharacter-closed over $k$ if and only if $M(k)\cdot \pi(v)$ is cocharacter-closed over $k$.
\item[(iii)]
Let $H \subseteq G$ be a connected subgroup. Then
$H \subseteq G_v$ if and only if $H \subseteq G_{\pi(v)}$.
\end{itemize}
\end{lem}

\begin{proof}
We first note that $V/F$ is defined over $k$ by \cite[2.2]{BR}, so the statement of the proposition makes sense. 

(i).
The forward direction is obvious.
So suppose $\lambda$ destabilizes $\pi(v)$.
By Remark \ref{rem:linear}, there is a $G$-equivariant closed embedding of $V$
in an $F\times G$-module $W$.
Let $\pi_W\colon W\ra W/F$ be the canonical projection.
We have an induced $G$-equivariant map $\phi$ from $V/F$ to $W/F$,
and $\lambda$ destabilizes $\phi(\pi(v))= \pi_W(v)$.
Hence it is enough to prove the result when $V= W$.
 
Clearly, we can take $G$ to be ${\rm Im}(\lambda)$.
Let $V_1= V_{\lambda,\geq 0}$ and let $V_2= V_{\lambda, <0}$.
Then $V= V_1\oplus V_2$ and $V_1$ and $V_2$ are both $F$-stable, since $F$ commutes with $G$.
The projection $V\ra V_2$ is $G$-equivariant and gives rise to a $G$-equivariant map $V/F\ra V_2/F$. Let $\pi_2\colon V_2\ra V_2/F$ be the canonical projection.
 
 The group $G$ acts on the dual space $V_2^*$.  
 Let $X_1,\ldots, X_m\in V_2^*$ be a basis such that each $X_i$ is a weight vector for $\lambda$. Since the weights of $\lambda$ on $V_2$ are all negative,
 these weights are all positive.  We can regard the $X_i$ as regular functions on $V_2$.
 Choose a generating set $f_1,\ldots, f_r$ for the ring of invariants $\ovl{k}[V_2]^F$.
 We can assume that each $f_j$ is a polynomial in the $X_i$ with no constant term and,
 since $\ovl{k}[V_2]^F$ is $G$-stable, we can assume also that each $f_j$ is a weight vector for $\lambda$.  Clearly each of these weights is positive.
 
Write $v= (v_1,v_2)$.
Since $\pi(v)$ is destabilized by $\lambda$, so is $\pi_2(v_2)$.
If $f \in \overline{k}[V_2]^F$ has weight $m>0$ with respect to $\lambda$, then
$f(\lambda(a) \cdot \pi_2(v_2) ) = (\lambda(a^{-1}) \cdot f)(\pi_2(v_2)) 
= a^{-m} f (v_2)$, and so $\lim_{a\to 0} \lambda(a)\cdot \pi_2(v_2)$ can only exist if $f(v_2)=0$.
It follows that $f_j(v_2)= 0= f_j(0)$ for all $j$.
This implies that $\pi_2(v_2)= \pi_2(0)$.
But $F$ is finite, so we must have $v_2= 0$.
Hence $v= (v_1,0)\in V_1$ and so $\lambda$ destabilizes $v$, as required.

(ii).
Suppose $M(k)\cdot v$ is cocharacter-closed over $k$.
Let $\lambda\in Y_k(M) \subseteq Y_k(G)$ such that $\lambda$ destabilizes $\pi(v)$.
Then $\lambda$ destabilizes $v$ by part (i), so there exists $g\in M(k)$ such that 
$\lim_{a\to 0} \lambda(a)\cdot v= g\cdot v$.
Now $g\cdot \pi(v)= \pi(g\cdot v)= \pi\left(\lim_{a\to 0} \lambda(a)\cdot v\right)= \lim_{a\to 0} \lambda(a)\cdot \pi(v)$.
It follows that $M(k)\cdot \pi(v)$ is cocharacter-closed over $k$.
 
 Conversely, suppose $M(k)\cdot v$ is not cocharacter-closed over $k$.
 Then $(F\times M)(k)\cdot v$ is not cocharacter-closed over $k$, by Corollary \ref{cor:conn_nonconn}. 
 Hence there exists $\lambda\in Y_k(M)$ such that $v':= \lim_{a\to 0} \lambda(a)\cdot v$ exists and is not $(F\times M)(k)$-conjugate to $v$.
 Now $\pi(v')= \lim_{a\to 0} \lambda(a)\cdot \pi(v)$.
 If $\pi(v')$ is $M(k)$-conjugate to $\pi(v)$, say $g\cdot \pi(v)= \pi(v')$, then $\pi(g\cdot v)= \pi(v')$, so $g\cdot v\in F\cdot v'$.
 But then $v'$ is $(F\times M)(k)$-conjugate to $v$, a contradiction.
 Hence $M(k)\cdot \pi(v)$ is not cocharacter-closed over $k$.

(iii).
We have $G_v \subseteq G_{\pi(v)}$, hence one assertion of (iii) is clear.
Conversely, suppose that $H \subseteq G_{\pi(v)}$ is a connected
subgroup. Then the orbit map $H \rightarrow H\cdot v$ has image in the finite set $F \cdot v$.
As $H$ is connected, this implies that $H \subseteq G_v$, as required.
\end{proof}

We are now in a position to prove Theorem \ref{thm:k-point}.

\begin{proof}[{Proof of Theorem \ref{thm:k-point}}]
By Lemma \ref{lem:k_s-point}, we may assume $v \in V(k_s)$.
Let $1=\gamma_1,\dots,\gamma_r \in \Gamma$ be chosen such that
$\{v,\gamma_2(v),\dots,\gamma_r(v)\}$ is the $\Gamma$-orbit
through $v$.
Set $\tuple{v} = (v, \gamma_2(v), \dots, \gamma_r(v)) \in V(k_s)^r$.
We first show that we may replace $v$ by $\tuple{v}$, where we 
consider the diagonal action of $M$ on $V^r$.
Let $\lambda \in Y_k(M)$. As $\gamma(\lambda) = \lambda$ for all $\gamma \in \Gamma$,
we have that $\lambda$ destabilizes $v$ if and only if it
destabilizes the tuple $\tuple{v}$.
If $\lambda$ destabilizes $v$ and $\lim_{a \to 0}\lambda(a) \cdot v = v'$,
then $\lim_{a\to 0} \lambda(a) \cdot \tuple{v} = (\gamma_1(v'),\dots,\gamma_r(v')) =: \tuple{v}'$.
For $g \in M(k)$, the equation $\gamma(g)=g$ for all $\gamma \in \Gamma$ implies
that $g\cdot v = v'$ if and only if $g\cdot\tuple{v} = \tuple{v}'$.
In particular, $M(k)\cdot v$ is cocharacter-closed over $k$ if and only if
$M(k)\cdot \tuple{v}$ is so.

Now let $H \subseteq G$ be a $k$-defined connected subgroup.
If $H \subseteq G_{\tuple{v}}$, then clearly $H \subseteq G_v$.
Conversely, suppose $H \subseteq G_v$ and let $h \in H(k_s)$.
As $\gamma(h) \in H(k_s)$ for all $\gamma \in \Gamma$, we get
$h \in G_{\tuple{v}}$. Since $H(k_s)$ is dense in $H$, this implies
that $H \subseteq G_{\tuple{v}}$.
It thus suffices to construct $W$ and $w$ and prove the assertions in (i) and (ii) for $\tuple{v}$.

Let $F = S_r$ be the symmetric group on $r$ letters, and let $F$ act
on $V^r$ by permuting the coordinates. Then the action of $F$ commutes
with the diagonal action of $G$, hence $F \times G$ acts on $V^r$.
As the assumptions of Lemma \ref{lem:finquot} are satisfied, we may
replace $V^r$ by $W = V^r/S_r$ and $\tuple{v}$ by $w = \pi(\tuple{v})
\in W(k_s)$. Now by construction, $\gamma(w) = w$ for all $\gamma \in \Gamma$,
hence $w$ is a $k$-point. This finishes the proof.
\end{proof}

\section{Geometric and rational conjugacy}
\label{sec:mainconj}

Our next theorem allows us to relate geometric $G$-conjugacy to
rational $R_u(P_\lambda)$-conjugacy, provided the
orbit has a $k$-defined stabilizer.
The result relies on Theorem \ref{thm:Ruconj}, 
but here we do not require $k$ to be perfect.

\begin{thm}
\label{thm:MCkdefinedstabilizer}
Let $v \in V$ and suppose that $G_v^0$ is $k$-defined.
Let $\lambda \in Y_k(G)$.
Suppose that $v' = \lim_{a \to 0} \lambda(a) \cdot v$ 
exists and is $G$-conjugate to $v$.
Then $v'$ is $R_u(P_\lambda)(k)$-conjugate to $v$.
\end{thm}

\begin{proof}
Let $g \in G$ satisfy $v' = g \cdot v$.
By Theorem \ref{thm:Ruconj}, we can find $u \in R_u(P_\lambda)$ such that
$v' = u^{-1} \cdot v$. Then $u \cdot \lambda$ is a cocharacter of $G_v^0$.

We first claim that $P_{u\cdot \lambda}(G_v^0)$ is $k$-defined.  
Indeed, by Proposition \ref{prop:smoothintersection} this group coincides with
the scheme-theoretic intersections
$G_v^0 \cap P_{u \cdot \lambda} = G_v^0 \cap P_\lambda$,
which are therefore smooth.
According to our assumptions both $G_v^0$ and $P_\lambda$ are $k$-defined,
and thus the (smooth) intersection is $k$-defined by
\cite[Prop.\ 12.1.5]{springer}.

By \cite[Thm.\ 18.2]{borel}, $P_{u \cdot \lambda}(G_v^0)$ contains a $k$-defined
maximal torus $S$, which automatically is a maximal torus of $G_v^0$.
Then $S$ is contained in some $k$-defined maximal torus $T$ of $P_{u\cdot \lambda}$.
Since $T$ is contained in a $k$-defined R-Levi subgroup of $P_{u\cdot \lambda}$
and all such subgroups are $R_u(P_\lambda)(k)$-conjugate,
we can find $x \in R_u(P_\lambda)(k)$ such that
$S \subseteq L_{x \cdot \lambda}$.
After replacing $\lambda$ with $x \cdot \lambda$ and $v'$ with $x \cdot v'$,
we may assume without loss of generality that $\lambda$ centralises $S$.
This forces $S$ to be contained in $G_{v'}$, since $v' = \lim_{a \to 0} \lambda(a) \cdot v$. By assumption, $G_{v'}$ is
$G$-conjugate to $G_v$ and thus has the same rank. In particular, $S$ is a maximal
torus of $G_{v'}$. But as $\lambda$ is a cocharacter of $G_{v'}$ that commutes with $S$,
we deduce that $\IM(\lambda) \subseteq S \subseteq G_v$.
Hence $v' = v$, which finishes the proof.
\end{proof}

We record a number of consequences of 
Theorem \ref{thm:MCkdefinedstabilizer}
which include
Theorem \ref{thm:Gvkdef}(i), (ii) (first assertion) and (iv).

\begin{cor}
\label{cor:MCkdefinedstabilizer}
Let $v \in V$ and suppose that $G_v^0$ is $k$-defined.
\begin{itemize}
\item[(i)] Let $\lambda \in Y_k(G)$.
Suppose that $v' = \lim_{a \to 0} \lambda(a) \cdot v$ exists and is $G$-conjugate to $v$.
Then $v'$ is $G(k)$-conjugate to $v$. 
\item[(ii)]
Suppose $G\cdot v$ is Zariski-closed. Then $G(k)\cdot v$ is cocharacter-closed
over $k$.
\item[(iii)]
Let $k'/k$ be an algebraic field extension and suppose that $G(k')\cdot v$ is
cocharacter-closed over $k'$. Then $G(k)\cdot v$ is cocharacter-closed over $k$.
\item[(iv)]
Let $S$ be a $k$-defined torus of $G_v$ and set $L:= C_G(S)$.
Let $\lambda\in Y_k(L)$ such that $v':= \lim_{a\to 0} \lambda(a)\cdot v$ exists and
$v'\in G\cdot v$.  Then $v'\in R_u(P_\lambda(L))(k)\cdot v$.
\item[(v)] Let $w \in V$ and suppose that
both $G(k)\cdot w \prec G(k)\cdot v$ and $G(k)\cdot v \prec G(k)\cdot w$.
Then $G(k) \cdot v= G(k)\cdot w$.
\end{itemize}
\end{cor}

\begin{proof}
Parts (i) and (ii) are immediate consequences of Theorem \ref{thm:MCkdefinedstabilizer}.
Part (iii) follows from (i) and the fact that $G(k')$-conjugacy implies $G$-conjugacy.

Part (iv) follows from Theorem \ref{thm:MCkdefinedstabilizer} and Proposition \ref{prop:ruplevi}.

For (v), 
by assumption, there exist cocharacters $\lambda_1,\dots,\lambda_n \in Y_k(G)$ and
elements $v_1, \dots,$ $v_{n+1} \in V$ such that
$v_1 = v$, $v_{n+1} = g \cdot v$ for some $g \in G(k)$, $v_j = g' \cdot w$ for some
$j$ and some $g' \in G(k)$, and
$\lim_{a \to 0}\lambda_i(a)\cdot v_i = v_{i+1}$ for $1\leq i\leq n$.
As $v_1$ and $v_{n+1}$ are $G$-conjugate, all of the $v_i$ are $G$-conjugate
(for if $v_i \not \in G \cdot v_{i-1}$, then, as $v_{i}$ lies in the closure
of $G \cdot v_{i-1}$, the orbit $G \cdot v_i$ has strictly smaller dimension than $G \cdot v_{i-1}$,
which forces all of the subsequent orbits $G \cdot v_i,\dots, G \cdot v_{n+1}$ to have smaller dimension than $G \cdot v$).
By (i), $v_2$ is $G(k)$-conjugate to $v_1$.
In particular, $G_{v_2}^0$ is again $k$-defined. Repeating the argument for $v_3$ and so on,
we find that $w$ is $G(k)$-conjugate to $v$, as required.
\end{proof}

\begin{rems} 
\label{rem:strongerrationality}
(i).
Note that Theorem \ref{thm:MCkdefinedstabilizer} fails without
the assumption on $G_v^0$.
E.g., see \cite[Rem.\ 5.10]{GIT} for the failure of
Corollary \ref{cor:MCkdefinedstabilizer}(ii) without this assumption. 

(ii).
Suppose that $G_v^0$ is $k$-defined and that 
$G(k) \cdot v$ is not cocharacter-closed over $k$.
By Theorem \ref{thm:MCkdefinedstabilizer}, there exists a $k$-defined
cocharacter $\lambda$ such that $v' = \lim_{a \to 0} \lambda(a) \cdot v$
exists and does not belong to $G \cdot v$.
In the terminology of \cite[\S4]{GIT}, $v$ is uniformly
$S$-unstable over $k$ for $S: = \overline{G \cdot v'}$, and $v$ does
not belong to $S$ (cf.\ the proof of Lemma~\ref{lem:geom_galois_ascent}).
In particular, there exist non-trivial cocharacters which
belong to the \emph{optimal class for $v$ with respect to $S$ over $k$},
see \cite[Thm.\ 4.5]{GIT}.

(iii).
Corollary \ref{cor:MCkdefinedstabilizer} refines 
the descent assertions in Theorem \ref{thm:leviascentdescent} and Proposition \ref{prop:galois_descent} (Levi and Galois descent),  
albeit under the stronger hypothesis that
$G_v^0$ is $k$-defined. 

(iv).
Corollary \ref{cor:MCkdefinedstabilizer}(i) is an instance where $G$-conjugacy
descends to $G(k)$-conjugacy. This property may be studied in cohomological terms as follows:
Suppose that $v,v' \in V(k)$ are $k$-points which are $G$-conjugate. This means that
$v' \in (G \cdot v)(k)$. Thus $G(k)$-conjugacy would follow
provided that the map $G(k) \rightarrow (G\cdot v)(k)$ is surjective.
By \cite[III, \S4, Cor.\ 4.7]{DG} this is automatic if
all $\mathcal{G}_v$-torsors over $k$ are trivial, where $\mathcal{G}_v$ is
the scheme-theoretic stabilizer of $v$ in $G$.
If we further assume that either $k$ is perfect or that the orbit map 
$G \rightarrow G\cdot v$ is separable, then the isomorphism classes
of $\mathcal{G}_v$-torsors over $k$ are parametrized by the 
Galois cohomology $H^1(\Gamma,\mathcal{G}_v)$, where $\Gamma = \Gal(k_s/k)$ (see \cite[III, \S5, Cor.\ 3.5,3.6]{DG}).
But even in the separable case, $H^1(\Gamma,\mathcal{G}_v)$ does not always vanish.
Our assumption that $v'$ arises as the limit along a $k$-defined cocharacter $\lambda$
allows us to circumvent all these technical difficulties and to avoid additional
assumptions.

(v).
Bremigan considers the case where
the field $k$ has characteristic $0$ and is complete 
under a non-trivial real absolute value, \cite{Bremigan}.
In \cite[Prop.\ 5.3]{Bremigan}, he proves that if $G \cdot v$ is
Zariski-closed, then $G(k) \cdot v$ is Hausdorff-closed (i.e., closed in the topology
induced by the extra structure on $k$). 
It follows that 
$G(k) \cdot v$ is cocharacter-closed over $k$, as the limit along
a $k$-defined cocharacter belongs to the Hausdorff-closure of $G(k) \cdot v$.
Hence, we obtain Corollary \ref{cor:MCkdefinedstabilizer}(ii) in this case.

(vi).
Corollary \ref{cor:MCkdefinedstabilizer}(v) does not imply that the relation
$\prec$ on all $G(k)$-orbits is antisymmetric. It does, however, assert that antisymmetry holds if one of the two orbits has a $k$-defined stabilizer.
\end{rems}

The following result shows that the assumption
in Theorem \ref{thm:MCkdefinedstabilizer}
that $G_v^0$ is $k$-defined
is automatic in many naturally-occurring cases (e.g., if $\Char(k)=0$).

\begin{prop} 
\label{prop:perfectorseparable}
Let $v \in V$ and 
assume that one of the following conditions is
satisfied:
\begin{itemize}
\item[(i)]
$v \in V(k)$ and $k$ is perfect;
\item[(ii)]
$v \in V(k_s)$, $G_v(k_s)$ is $\Gamma$-stable and $G\cdot v$ is separable.
\end{itemize}
Then $G_v^0$ is $k$-defined. 
In particular, 
the assertions of Theorem \ref{thm:MCkdefinedstabilizer} and
Corollary \ref{cor:MCkdefinedstabilizer} hold for $v$.
\end{prop}

\begin{proof}
Case (i) is \cite[Prop.\ 12.1.2]{springer}.
For (ii), suppose that $v\in V(k_s)$, that $G_v(k_s)$ is $\Gamma$-stable and
that $G \cdot v$ is separable. Then $G_v^0$ is $k$-defined. Indeed, separability implies
that $G_v$ is $k_s$-defined (see \cite[\emph{loc.\ cit.}]{springer}).
Due to the assumption on $G_v(k_s)$, the Galois criterion
for $k$-definedness implies that $G_v$ is $k$-defined, hence so is $G_v^0$.
\end{proof}

\begin{rem}
\label{rem:perfectorseparable}
There are situations where $G_v^0$ is $k$-defined for all
$k$-points $v \in V(k)$. For instance, Proposition \ref{prop:perfectorseparable} 
above implies that
this is the case whenever $k$ is perfect or all $G(k)$-orbits are separable.
In this case $\prec$ is antisymmetric on the set $\{ G(k)\cdot v\mid v \in V(k)\}$,
hence is a partial order on this set.
We do not know whether antisymmetry holds in general.
See also Proposition \ref{prop:anti} below.
\end{rem}

We give two simple examples.
The first example illustrates that having $G_v$ $k$-defined does
not imply separability in general 
(cf.\ Proposition \ref{prop:perfectorseparable}(ii)).
The second example shows that even for separable orbits, the
converse of Corollary \ref{cor:MCkdefinedstabilizer}(ii) is false in general.

\begin{exs} 
\label{ex:SL2,GL2}
(i).
Let $G = \SL_2$, $\Char(k)=2$, and consider the adjoint action of 
$G$ on $V = \mathfrak{sl}_2$.
Let $v = \begin{pmatrix}0 &1 \\ 0&0\end{pmatrix} \in V(k)$.
Then $G_v = \left\{\begin{pmatrix}1 &a \\ 0&1\end{pmatrix} 
\:\: \vline \:\: a \in k\right\}$ is $k$-defined, 
but $G \cdot v$ is not separable.

(ii).
 Let $G= {\rm PGL}_2$, $\Char(k)=2$, acting on $V= \pgl_2$ by conjugation.  Let $\ovl{A}$ denote the image in ${\rm PGL}_2$ of $A\in {\rm GL}_2$, and likewise for images of elements of $\gl_2$.  Define 
 $v \in V(k)$ by $v:= \overline{\left(
 \begin{array}{cc}
  0 & 1 \\
  x^2 & 0
 \end{array}
 \right)},$ where $x^2\in k$ but $x\not\in k$.  Then $v= g\tilde{v}g^{-1}$, where $g:= \overline{\left(
 \begin{array}{cc}
  1 & 0 \\
  x & 1
 \end{array}
 \right)}$ and $\tilde{v}:=
 \overline{\left(
 \begin{array}{cc}
  x & 1 \\
  0 & x
 \end{array}
 \right)} =  \overline{\left(
 \begin{array}{cc}
  0 & 1 \\
  0 & 0
 \end{array}
 \right)}.$
 Define $\lambda\in Y_k(G)$ by $\lambda(a)=  \overline{\left(
 \begin{array}{cc}
  a & 0 \\
  0 & a^{-1}
 \end{array}
 \right)}$.  Then $\lim_{a\to 0} \lambda(a)\cdot \tilde{v}=  0$,
so $G\cdot \tilde{v}$ is not closed and hence $G\cdot v$ is not closed.
Moreover, $G \cdot v$ is separable.
However, it is easily checked that $G(k)\cdot v$ is cocharacter-closed over $k$ (the unique Borel subgroup whose Lie algebra contains $v$ is not $k$-defined).
Note that $0$ lies in $\overline{G \cdot v}\setminus G \cdot v$, so $\overline{G \cdot v}\setminus G \cdot v$ does indeed have a $k$-point.
\end{exs}

\begin{rems}
\label{rems:hoff}
(i).
The non-separability of the orbit map $G \rightarrow G \cdot v$ in Examples \ref{ex:SL2,GL2}(i)
is because $\zz(\gg)$ is non-zero for $G=\SL_2$ in characteristic $2$,
whereas the group-theoretic centre $Z(G)$ vanishes.
However, as $\zz(\gg)$ consists of semisimple elements, this does not
affect the separability of the $R_u(P_\lambda)$-orbit through $v$.

In general, let $v \in V(k_s)$ such that $G_v(k_s)$ is $\Gamma$-stable,  
and suppose $\lambda \in Y_k(G)$ is such that $v':=\lim_{a\to0}\lambda(a)\cdot v$
exists and lies in $G\cdot v$. 
Further suppose that the orbit $R_u(P_\lambda)\cdot v$ is separable.
Then we may still conclude that $v' \in R_u(P_\lambda)(k)\cdot v$,
i.e., the assertion of Theorem \ref{thm:MCkdefinedstabilizer} holds
in this case.  Indeed, according to Theorem \ref{thm:Ruconj}, $v'\in R_u(P_\lambda)\cdot v$,
and $v'$ is a $k_s$-point.
Clearly, $H^1(\Gal(k_s/k_s),R_u(P_\lambda)_v)$ vanishes.
Due to the separability assumption, this means
that the map $R_u(P_\lambda)(k_s) \rightarrow (R_u(P_\lambda)\cdot v)(k_s)$
is surjective (see Remark \ref{rem:strongerrationality}(iv)),
hence $v' \in R_u(P_\lambda)(k_s) \cdot v$.
By Theorem \ref{thm:rupgalois}, 
$v' \in R_u(P_\lambda)(k) \cdot v$.

(ii).
Hoffmann asked the following question: let $v \in V(k)$ and suppose
that $\overline{G \cdot v} \setminus G \cdot v$ has a $k$-point. Does there
exist $\lambda \in Y_k(G)$ such that $v': = \lim_{a\to 0}\lambda(a) \cdot v$
exists and $v' \not\in G \cdot v$?
Clearly, the answer is no if $G(k) \cdot v$ is cocharacter-closed over $k$.
Example \ref{ex:SL2,GL2}(ii) shows that this is possible under the present 
hypotheses, even if we assume $G \cdot v$ is separable.
If we assume that $G(k)\cdot v$ is not cocharacter-closed over $k$ and that
$G_v^0$ is $k$-defined, the answer is yes, by Corollary \ref{cor:MCkdefinedstabilizer}(i).
\end{rems}

If $v'= \lim_{a\to 0} \lambda(a)\cdot v$ then we have seen there is a complicated relationship between the property that $v'$ lies in $G(k)\cdot v$ and the property that $v'$ lies in $R_u(P_\lambda)(k)\cdot v$.  If the latter holds then we can, for instance, apply Theorem~\ref{thm:rupgalois} and Proposition~\ref{prop:ruplevi}.  The following open question seems technical but resolving it is crucial to gaining a more complete understanding of the behavior of the $G(k)$-orbits. 

\begin{qn}
\label{qn:main}
Let $v\in V$ and $\lambda \in Y_k(G)$ 
such that $v':=\lim_{a \to 0} \lambda(a) \cdot v$ exists. 
Suppose that $v'$ is $G(k)$-conjugate to $v$. 
Is $v'$ then $R_u(P_\lambda)(k)$-conjugate to $v$?
\end{qn}

\begin{rems}
\label{rem:Ruconj}
(i). Theorem \ref{thm:Ruconj} implies that Question \ref{qn:main} 
has a positive answer
whenever $k$ is perfect.

(ii). By Theorem \ref{thm:MCkdefinedstabilizer},
Question \ref{qn:main} has a positive answer whenever $G_v^0$ is $k$-defined.
In particular, see Proposition \ref{prop:perfectorseparable} 
for separability assumptions that are sufficient to imply this. 

(iii).
Corollary \ref{cor:mainconjcocharclosed}
shows that 
Question \ref{qn:main} has an affirmative answer
for cocharacter-closed orbits.

(iv).
In the context of $G$-complete reducibility, the fact that
Question \ref{qn:main} has a positive answer for algebraically closed
fields has been used for instance
by Stewart (\cite[Cor.\ 3.6.2]{stewart}) and
Uchiyama (\cite[Prop.\ 3.6]{uchiyama}).  For another application, see \cite[Cor.~3.5]{GIT}.
\end{rems}

In order to study the question of whether $\prec$ is antisymmetric,
we also state the following weaker version of Question \ref{qn:main}.

\begin{qn}
\label{qn:weak}
Let $v\in V$ and $\lambda \in Y_k(G)$ 
such that $v':=\lim_{a \to 0} \lambda(a) \cdot v$ exists. 
Suppose that $v'$ is $G(k)$-conjugate to $v$. 
Is $v'$ then $G^0(k)$-conjugate to $v$?
\end{qn}

A positive answer to this question is related to antisymmetry as follows.

\begin{prop}
\label{prop:anti}
Suppose that the assertion of Question \ref{qn:weak} holds 
for any reductive group acting on an affine variety. Then the preorder
$\prec$ on $G(k)$-orbits is antisymmetric, hence is a partial order.
\end{prop}

\begin{proof}
Suppose we have $v,v' \in V$ with $G(k) \cdot v' \prec G(k) \cdot v \prec G(k) \cdot v'$.
By Lemma \ref{lem:transitiveclosure} we can find
$v_1,\dots,v_n \in V$, $\lambda_1,\dots,\lambda_n \in Y_k(G)$ such that
all limits $\lim_{a\to 0}\lambda_i(a) \cdot v_i$ ($1 \leq i \leq n$) exist,
$v=v_1$, $v_i = \lim_{a\to 0}\lambda_{i-1}(a) \cdot v_{i-1}$ for
$2 \leq i \leq n$, $v' = g' \cdot v_j$ for some $1 \leq j \leq n$, $g' \in G(k)$ and
$\lim_{a \to 0}\lambda_n(v_n) = g \cdot v$ for some $g \in G(k)$.
Now let $\tuple{v}=(v_1,\dots,v_n) \in V^n$ and let
$\tuple{G} = C_n \ltimes G^n$.
Here $C_n = \langle \sigma \rangle$ denotes the cyclic group of order $n$ acting on $G^n$ by permuting the indices cyclically ($\sigma(i)= i+1\ {\rm mod}\ n$).
Let $\tuple{\lambda} = (\lambda_1,\dots,\lambda_n)
\in Y_k(G^n) = Y_k(\tuple{G})$.

Then $\tuple{G}$ acts naturally on $V^n$, and
\begin{align*}
\lim_{a\to 0}\tuple{\lambda}(a) \cdot \tuple{v} &=
(\lim_{a\to 0} \lambda_1(a)\cdot v_1, \dots, \lim_{a\to 0} \lambda_n(a) \cdot v_n) \\
&= (v_2, \dots, v_n, g \cdot v)
= ((1,\dots,1,g)\sigma) \cdot \tuple{v} \in \tuple{G}(k) \cdot \tuple{v}.
\end{align*}

Since we assume that Question \ref{qn:weak} has an affirmative answer for $\tuple{G}$, we can find
$\tuple{u} = (u_1,\dots,u_n) \in \tuple{G}^0(k) = G(k)^n$
such that
$\tuple{u} \cdot \tuple{v} = (v_2,\dots,v_n, g\cdot v)$.
Hence
$v' = g' \cdot v_j = (g'u_{j-1}\cdots u_1) \cdot v \in G(k) \cdot v$.
We conclude that $G(k) \cdot v = G(k) \cdot v'$.
\end{proof}

\begin{rem}
Note that the proof of Proposition \ref{prop:anti} requires that
Question \ref{qn:weak} has an affirmative answer for all groups of the form
$C_n \ltimes G^n$ acting on $V^n$, $n \in \NN$. In particular, even though
the assertion of Question \ref{qn:weak} holds trivially for connected $G$, this does
not imply the antisymmetry of $\prec$ for connected $G$.
\end{rem}

In order to obtain a converse of Proposition \ref{prop:anti}, we
need the following lemma.

\begin{lem}
\label{lem:connectedaccessible}
Let $v \in V$, $\lambda \in Y_k(G)$.  Suppose that
$v' = \lim_{a \to 0} \lambda(a) \cdot v$ exists and $v' = g \cdot v$
for some $g \in G(k)$.
Then $G^0(k) \cdot v$ is accessible from $G^0(k) \cdot v'$.
\end{lem}

\begin{proof}
As $G/G^0$ is finite, there exists some $n\in \NN$ such that $g^n \in G^0(k)$.
Taking the limit along $g \cdot \lambda$ maps $v'=g \cdot v$
to $g \cdot v' = g^2 \cdot v$, and then taking the limit along $g^2 \cdot \lambda$ maps 
$g^2 \cdot v$ to $g^2 \cdot v' = g^3 \cdot v$.
Continuing in this way, we see that the orbit through $g^n \cdot v$ is 
accessible from $v'$. As $g^n \in G^0(k)$, this implies that
$G^0(k) \cdot v$ is accessible from $G^0(k) \cdot v'$, as claimed.
\end{proof}

\begin{prop}
\label{prop:anti2}
Suppose that the preorder $\prec$ is antisymmetric on the set of 
$G^0(k)$-orbits in $V$. Then the assertion of Question \ref{qn:weak}
holds for $G$.
\end{prop}

\begin{proof}
Let $v \in V$, $\lambda \in Y_k(G)$ and suppose that
$v' = \lim_{a \to 0} \lambda(a) \cdot v$ exists and $v' = g \cdot v$
with $g \in G(k)$.
By Lemma \ref{lem:connectedaccessible}, $G^0(k) \cdot v$ is accessible
from $G^0(k) \cdot v'$. Conversely, it is immediate that $G^0(k) \cdot v'$ is 
accessible from $G^0(k) \cdot v$.
By the antisymmetry of $\prec$, we conclude that $G^0(k) \cdot v =
G^0(k) \cdot v'$. Hence $v' \in G^0(k) \cdot v$, as required.
\end{proof}

Combining Propositions \ref{prop:anti} and \ref{prop:anti2} we obtain
the following result.

\begin{cor}
\label{cor:anti}
The following are equivalent (for fixed $k$):
\begin{itemize}
\item[(i)] The assertion of Question \ref{qn:weak} holds for any reductive group acting
on any affine variety;
\item[(ii)] the preorder $\prec$ is antisymmetric for any reductive group acting on
any affine variety;
\item[(iii)] the preorder $\prec$ is antisymmetric for any connected reductive group acting
on any affine variety.
\end{itemize}
\end{cor}

\section{Reduction to ${\rm GL}_n$}
\label{sec:reduction}

In this section we use Theorem~\ref{thm:MCkdefinedstabilizer} 
to prove a result which reduces certain questions involving 
$G(k)$-orbits to the special case when $G= {\rm GL}_n$ for some $n$.  
We start by recalling a standard construction from GIT, cf.~\cite[5.3]{BR}.  
Let $G,M$ be reductive $k$-groups with $G\subseteq M$ and let $V$ 
be an affine $G$-variety over $k$.  We define an action of $G$ on 
$M\times V$ by $g\cdot (m,v) = (mg^{-1}, g\cdot v)$.  
Write $M\times^G V$ for the quotient of $M\times V$ by this $G$-action 
and let $\pi_G\colon M\times V\ra M\times^G V$ be the canonical projection.  
Then $\pi_G$ is $k$-defined \cite[2.2]{BR}.  
The group $M$ acts on $M\times V$ by left multiplication on the 
first factor and trivially on the second factor; this descends 
to give an action of $M$ on $M\times^G V$.  If $G= M$ then there 
is an obvious $k$-defined $G$-equivariant isomorphism from $G\times^G V$ to $V$.  
Projection onto the first factor gives a $k$-defined $M$-equivariant 
morphism $\eta\colon M\times^G V\ra M/G$, where $M$ acts on the coset 
space $M/G$ by left multiplication.

\begin{thm}
\label{thm:reduction}
 Let $G,M$ be reductive $k$-groups with $G\subseteq M$ 
and let $V$ be an affine $G$-variety over $k$.  
Let $k'/k$ be an extension of fields.  
Set $W= M\times^G V$ and define $\phi\colon V\ra W$ by $\phi(v)= \pi_G(1,v)$ (note that $\phi$ is a $k$-defined $G$-equivariant closed embedding: see \cite{BR}, for example).  
Then the following hold:
\begin{itemize}
\item[(i)]  
for all $v\in V$ and all $\lambda\in Y_{k'}(M)$, if $\lambda$ destabilizes $\phi(v)$ then $u\cdot \lambda\in Y_{k'}(G)$ for some $u\in R_u(P_\lambda(M))(k')$;
\item[(ii)]  
for all $v\in V$ and all $\lambda\in Y_{k'}(G)$, if $v':= \lim_{a\to 0}\lambda(a)\cdot v$ exists and $m\cdot \phi(v) = \phi(v')$ for some $m\in M(k')$ then $m\in G(k')$;
\item[(iii)]  
$M_{\phi(v)} = G_{\phi(v)}= G_v$;
\item[(iv)] 
for all $v\in V$ and all $\lambda\in Y_{k'}(G)$, $\lambda$ destabilizes $v$ if and only if $\lambda$ destabilizes $\phi(v)$;
\item[(v)]  
for all $v\in V$ and all $\lambda\in Y_{k'}(G)$, $\lambda$ properly destabilizes $v$ over $k'$ for $G$ if and only if $\lambda$ properly destabilizes $\phi(v)$ over $k'$ for $G$ if and only if $\lambda$ properly destabilizes $\phi(v)$ over $k'$ for $M$;
\item[(vi)]  
for all $v\in V$, $G(k')\cdot v$ is cocharacter-closed over $k'$ if and only if $G(k')\cdot \phi(v)$ is cocharacter-closed over $k'$ if and only if $M(k')\cdot \phi(v)$ is cocharacter-closed over $k'$.
\end{itemize}
\end{thm}

\begin{proof}
 By extending scalars, we can regard $G$ and $M$ as $k'$-groups and $V$ as a $k'$-defined $G$-variety.  Hence we can assume without loss that $k'= k$.
 
 Now let $v\in V$, let $\lambda\in Y_k(M)$ and suppose $\lambda$ 
destabilizes $\phi(v)$.  Let $x$ be the coset $G\in M/G$.  
Since the map $\eta\colon M\times^G V\ra M/G$ induced by projection onto the first factor is $M$-equivariant, 
$\lambda$ destabilizes $\eta(\phi(v))= x$.  
Set $x'= \lim_{a\to 0} \lambda(a)\cdot x$.  
The orbit $M\cdot x$ is closed---it is the whole of $M/G$---and $M_x= G$ 
is $k$-defined, so $M(k)\cdot x$ is cocharacter-closed over $k$, 
by Corollary~\ref{cor:MCkdefinedstabilizer}(ii).  
Hence by Corollary~\ref{cor:mainconjcocharclosed}, 
there exists $u\in R_u(P_\lambda(M))(k)$ such that $x= u\cdot x'$.  
Then $u\cdot \lambda$ fixes $x$, so $u\cdot \lambda$ is a cocharacter 
of $M_x= G$.  This proves (i).  If we assume in addition that $\lambda$ 
is already a cocharacter of $G$ then $x'= x$; hence if $m\in M(k)$ 
and $m\cdot \phi(v) = \phi(v')$ 
then $m\cdot x= x$ by the $M$-equivariance of $\eta$, so $m\in G(k)$ 
and (ii) follows.  This also proves part (iii) (take $\lambda= 0$ in (ii)).
 
Parts (iv)--(vi) now follow immediately from (i)--(ii).
\end{proof}

We now see that to prove Galois ascent and to answer Question~\ref{qn:main}, 
it suffices to consider the special case $G= \GL_n$.

\begin{cor}
\label{cor:reduction}
\begin{itemize}
\item[(i)] 
If the answer to Question~\ref{qn:main} is yes when $G= \GL_n$ 
then the answer is yes for all $G$.
\item[(ii)]  
If Galois ascent holds when $G= \GL_n$ then it holds for all $G$.
\end{itemize}
\end{cor}

\begin{proof}
 We can embed $G$ as a $k$-defined subgroup of some $\GL_n$.  
Define $\phi$ as in Theorem~\ref{thm:reduction}.  
Part (i) now follows from Theorem~\ref{thm:reduction}.  
For part (ii), suppose $G(k_s)\cdot v$ is not cocharacter-closed over $k_s$.  
We want to show that $G(k)\cdot v$ is not cocharacter-closed over $k$.  
Now $\GL_n(k_s)\cdot \phi(v)$ 
is not cocharacter-closed over $k_s$, by Theorem~\ref{thm:reduction}(vii) 
(with $k'= k_s$).  By assumption, Galois ascent holds for $\GL_n$, 
so $\GL_n(k)\cdot \phi(v)$ is not cocharacter-closed over $k$.  
The result now follows from Theorem~\ref{thm:reduction}(vii) (with $k'= k$).
\end{proof}

\section{$G$-complete reducibility}
\label{sec:Gcr}

In this section we apply our previous results 
to the theory of $G$-completely reducible 
subgroups over $k$.  This was the original motivation for much of 
our work, and is an important source of examples.  The notion of a $G$-completely reducible subgroup was introduced by Serre \cite{serre2}; there are applications to the subgroup structure of reductive groups \cite{liebeckseitz0}, \cite{liebeckseitz}, spherical buildings \cite{serre1.5} and the Langlands correspondence \cite[Sec.~13]{lafforgue}.  For more details, see \cite{BMR} and \cite{serre2}.

First we recall 
the relevant definitions, then we explain the link with geometric 
invariant theory.

\begin{defn}
\label{def:gcr}
A subgroup $H$ of $G$ is said to be \emph{$G$-completely reducible ($G$-cr)}
if whenever $H$ is contained in an R-parabolic subgroup $P$ of $G$,
there exists an R-Levi subgroup of $P$ containing $H$.
Similarly, a subgroup $H$ of $G$ is said to be
\emph{$G$-completely reducible over $k$}
if whenever $H$ is contained in a $k$-defined
R-parabolic subgroup $P$ of $G$,
there exists a $k$-defined R-Levi subgroup of $P$ containing $H$.
\end{defn}

In \cite[Cor.\ 3.7]{BMR}, we show that $G$-complete 
reducibility has a geometric
interpretation in terms of the action of $G$ on $G^n$, the $n$-fold
Cartesian product of $G$ with itself, by simultaneous conjugation.

Let $\mathbf{h} \in G^n$
and let $H$ be the algebraic subgroup of $G$ generated by $\tuple{h}$,
i.e., by the components of the $n$-tuple $\tuple{h}$.
Then $G\cdot\mathbf{h}$ is closed in $G^n$ if and only if
$H$ is $G$-cr \cite[Cor.\ 3.7]{BMR}.
To generalize this to subgroups that are not topologically finitely
generated, we employed the following concept \cite[Def.\ 5.4]{GIT}.

\begin{defn}
\label{def:generictuple}
Let $H$ be a subgroup of $G$ and let $G\hookrightarrow\GL_m$
be an embedding of algebraic groups.
Then $\tuple{h} \in H^n$ is called a
\emph{generic tuple of $H$ for the embedding $G\hookrightarrow\GL_m$}
if $\tuple{h}$ generates the associative subalgebra of $\Mat_m$
spanned by $H$. We call $\tuple{h}\in H^n$ a \emph{generic tuple of $H$}
if it is a generic tuple of $H$ for some embedding $G\hookrightarrow\GL_m$.
\end{defn}

Note that generic tuples exist for any embedding $G\hookrightarrow\GL_m$ 
provided $n$ is sufficiently large. 
We give a characterization of $G$-complete reducibility over $k$ 
in terms of geometric invariant theory.

\begin{thm}
\label{thm:crvscocharclosed}
Let $H$ be a subgroup of $G$ and let $\tuple{h}\in H^n$ 
be a generic tuple of $H$.
Then $H$ is $G$-completely reducible over $k$ if and only if
$G(k) \cdot \tuple{h}$ is cocharacter-closed over $k$.
\end{thm}

\begin{proof}
 This follows from the proof of \cite[Thm.\ 5.9]{GIT} (which is the special case when $G$ is connected), replacing \cite[Thm.\ 3.10]{GIT} with Corollary~\ref{cor:conn_nonconn}.
\end{proof}

\begin{rem}
\label{rem:ktuple}
 Suppose the subgroup $H$ of $G$ is $k$-defined.  We can pick a generic tuple $\tuple{h}\in H(k_s)^n$ of $H$ for some $n$, and without loss we can assume that $\Gamma$ permutes the entries of the tuple.  Let $W= H^n/S_n$, where $S_n$ acts by permuting the entries of tuples, and let $\pi_W\colon H^n\ra W$ be the canonical projection.  Then $w:= \pi_W(\tuple{h})$ is a $k$-point, and it follows from Lemma~\ref{lem:finquot} and Theorem~\ref{thm:crvscocharclosed} that $H$ is $G$-cr over $k$ if and only if $G(k)\cdot w$ is cocharacter-closed over $k$.
\end{rem}

With this characterization in hand, we may apply our earlier rationality
results. Recall (cf.\ \cite[Def.\ 2.11]{herpel})
that a prime $p$ is called \emph{pretty good} for $G$
provided it is a good prime with the extra property that both
the root lattice in the character group and the coroot lattice in the
cocharacter group have no $p$-torsion. A prime $p$ which is very good
for $G$ is automatically pretty good, but the converse fails in general.

\begin{cor} 
\label{cor:prettycrdescent}
Suppose $p = \Char(k)$ is a pretty good prime for $G$ and let $H$ be a $k$-defined
subgroup of $G$. Let $k'/k$ be an algebraic field extension.
If $H$ is $G$-completely reducible over $k'$, then $H$ is $G$-completely reducible over $k$.
\end{cor}

\begin{proof}
Let $\tuple{h}$ be a generic tuple of $H$.
By \cite[Cor.\ 3.4]{BMRT}, the orbit 
$G \cdot \tuple{h}$ is separable, provided that $p$
is a very good prime for $G$.
It follows from \cite[Thm.\ 1.1]{herpel} that the condition may be relaxed
to require only that $p$ is pretty good for $G$.
Moreover, since $H$ is $k$-defined, $H(k_s)$ is dense in $H$,
so we can choose $\tuple{h} \in H(k_s)^n$ such that $\tuple{h}$ is a generic tuple of $H$. 
Then $G_{\tuple{h}}(k_s)= C_G(H)(k_s)$ is $\Gamma$-stable.
Hence the conditions of Proposition
\ref{prop:perfectorseparable}(ii) are satisfied, and the result
follows from Theorem~\ref{thm:crvscocharclosed} together with
Corollary \ref{cor:MCkdefinedstabilizer}(iii).
\end{proof}

\begin{rem}
The assertion of Corollary~\ref{cor:prettycrdescent} is false for inseparable
field extensions without the assumption on $p=\Char(k)$: see
\cite[Ex.\ 7.22]{BMRT} for an example in $G_2$ with $p=2$ and
\cite[Thm.\ 1.10]{uchiyama} for an example in $E_7$ with $p=2$.  Corollary~\ref{cor:prettycrdescent} shows that this kind of pathology can occur only in small characteristic.
\end{rem}

Theorem~\ref{thm:descent+ascent} and
Remark~\ref{rem:ktuple}
immediately yield the following results on 
Galois ascent/ descent and Levi ascent/descent for 
$G$-complete reducibility.
We note that we are able to obtain 
results of this nature
under slightly different hypotheses using a building-theoretic approach: see \cite{BMR:sepext}, \cite{BHMR:buildings}.

\begin{cor}
\label{cor:descent+ascentGCR}
Let $H$ be a $k$-defined subgroup of $G$ such that $C_G(H)$ is $k$-defined.
Then the following hold.
\begin{itemize}
\item[(i)] For any separable algebraic extension $k'/k$, 
$H$ is $G$-completely reducible over $k'$ if and only if 
$H$ is $G$-completely reducible over $k$.
\item[(ii)] For a $k$-defined torus $S$ of $C_G(H)$ let 
$L = C_G(S)$. Then 
$H$ is $G$-completely reducible over $k$ if and only if 
$H$ is $L$-completely reducible over $k$.
\end{itemize}
\end{cor}

The Levi ascent/descent statement of 
Corollary \ref{cor:descent+ascentGCR}(ii)
provides a different proof---valid for non-connected $G$---of 
Serre's result \cite[Prop.\ 3.2]{serre1.5}.

The next result has been obtained by McNinch in \cite[Prop.\ 4.1.1]{mcninch3}.
We give a different proof, which relies on Theorem \ref{thm:MCkdefinedstabilizer}.

\begin{cor}
\label{cor:linreductive}
Let $H$ be a $k$-defined linearly reductive subgroup of $G$.  
Then $H$ is $G$-completely reducible over $k$.
\end{cor}

\begin{proof}
Since $H$ is $k$-defined,
as in the proof of Corollary \ref{cor:prettycrdescent} we may choose a generic tuple
$\tuple{h} \in H(k_s)^n$ of $H$ such that $G_{\tuple{h}}(k_s)$ is $\Gamma$-stable.
By Theorem~\ref{thm:crvscocharclosed}, it suffices to show that
$G(k)\cdot \tuple{h} \subseteq G^n$ is cocharacter-closed over $k$.
By \cite[Lem.\ 2.4]{BMR2} and \cite[Thm.\ 5.8(iii)]{GIT},
the orbit $G \cdot \tuple{h}$ is closed.

We have $\cc_\gg(\tuple{h}) = \cc_\gg(H) = \Lie(C_G(H)) = \Lie(C_G(\tuple{h}))$,
where the second equation follows from
\cite[Lem.\ 4.1]{rich1}, and the other equations from
\cite[Lem.\ 5.5]{GIT} and its proof.
Thus the orbit $G \cdot \tuple{h}$ is separable. 
Therefore, 
it follows from
Corollary \ref{cor:MCkdefinedstabilizer}(ii)
and 
Proposition \ref{prop:perfectorseparable}(ii)
that $G(k) \cdot \tuple{h}$ is cocharacter-closed over $k$,
which finishes the proof.
\end{proof}

\begin{rem}
\label{rem:linreductive}
Observe that Corollary \ref{cor:linreductive}
is false if $H$ is not $k$-defined: e.g., 
let $G$ be $k$-split and let $S$ be a maximal $k$-defined torus.
Suppose $\lambda \in Y_k(S)$ and 
$u \in R_u(P_\lambda) \setminus R_u(P_\lambda)(k)$ are
chosen so that $H = u \cdot S$ is a torus which is no longer
$k$-defined. Clearly, $H \subseteq P_\lambda$.
But $H$ is not contained in any $k$-defined Levi subgroup.
Indeed, if $L = L_{x \cdot \lambda}$ with 
$x \in R_u(P_\lambda)(k)$ is such a Levi subgroup, then it
must coincide with $L_{u \cdot \lambda}$, which is the unique
Levi subgroup of $P_\lambda$ containing $H$. But this
implies that $u = x \in R_u(P_\lambda)(k)$, a contradiction.
\end{rem}

\begin{rem}
Following McNinch (see \cite{mcninch}), a subalgebra $\hh$ of $\gg$
is called \emph{$G$-completely reducible} provided that whenever 
$\hh$ is contained in $\Lie(P)$ for $P$ an R-parabolic subgroup of $G$,
there exists an R-Levi subgroup $L$ of $P$ with $\hh \subseteq \Lie(L)$.
All of the concepts and results of this section may be formulated and proved
for Lie algebras as well. For a generic tuple associated to $\hh$ we simply
take a generating tuple $\tuple{h} \in \hh^n$ for a suitable $n$.
See also \cite[\S5.3]{GIT}.
\end{rem}

\section{The action of ${\rm GL}(W)$ on ${\rm End}(W)$}
\label{sec:GLn}

In this section we illustrate the notion of cocharacter-closedness
in the classical context of conjugacy of endomorphisms under the general linear group.
Let $W_0$ be a finite-dimensional $k$-vector space with associated
$\overline{k}$-space $W = \overline{k} \otimes_k W_0$.
Consider $G = \GL(W)$ with its natural action by conjugation on $V = \End(W)$,
and let both $G$ and $V$ be endowed with the $k$-structures induced from $W_0$,
so that the action is $k$-defined.
For $f \in V(k) = \End(W_0)$, let $\mu_f \in k[T]$ be the minimal polynomial of $f$, and
$\chi_f \in k[T]$ the characteristic polynomial.
Note that both $\chi_f$ and $\mu_f$ remain invariant under field extensions of $k$.
For any $\mu \in k[T]$, there exists $f \in V(k)$ 
having $\mu$ as its minimal polynomial and its characteristic polynomial
(e.g., via considering the companion matrix of $\mu$, see \cite[Thm.\ 7.12]{Roman}).

In general, the minimal polynomial $\mu_f$ only encodes partial information about the
rational normal form of $f$ and hence about the orbit $G(k) \cdot f = \GL(W_0) \cdot f$.
However, our next
result shows that one can read off from $\mu_f$ whether $G(k) \cdot f$ is
cocharacter-closed.
It is convenient to consider $W_0$ as a $k[T]$-module, where $T$ acts via $f$.

\begin{prop}
\label{prop:cc_for_GL}
The following are equivalent:
\begin{itemize}
\item[(i)] The orbit $G(k)\cdot f \subseteq V$ is cocharacter-closed over $k$;
\item[(ii)] $\mu_f$ is square-free in $k[T]$ (i.e., has no repeated irreducible factors);
\item[(iii)] $W_0$ is semisimple as a $k[T]$-module.
\end{itemize}
\end{prop}

\begin{proof}
Suppose first that $W_0 = W_1 \oplus W_2$ decomposes as a $k[T]$-module.
Let $f_i = f|_{W_i}$ and let $G_i = \GL(\overline{k}\otimes W_i)$ ($i=1,2$).
Then by an application of Theorem \ref{thm:leviascentdescent}(ii),
$G(k)\cdot f$ is cocharacter-closed if and only if
$G_1(k) \cdot f_1$ and $G_2(k) \cdot f_2$ are cocharacter-closed.
Moreover, $\mu_f = \text{lcm}(\mu_{f_1},\mu_{f_2})$ 
(cf.\ \cite[Thm.\ 7.7]{Roman})
is
square-free if and only if both $\mu_{f_1}$ and $\mu_{f_2}$ are.
Thus we may assume from the outset that $W_0$ is an indecomposable $k[T]$-module.

In this case, it 
follows (e.g., from the primary cyclic decomposition of $W_0$,
see \cite[Thm.\ 7.6]{Roman})
that $\chi_f = \mu_f$ is the power of an
irreducible polynomial in $k[T]$.
Hence $\mu_f$ is square-free if and only if $\chi_f$ is irreducible,
which is easily seen to be equivalent to $W_0$ being irreducible as a $k[T]$-module.
This proves the equivalence of (ii) and (iii).

To prove that (i) is equivalent to (iii),
suppose first that $W_0$ is irreducible. Then $f$ stabilizes
no proper subspace of $W_0$, and hence $f$ is not destabilized
by any non-central $k$-defined cocharacter of $G$. Thus $G(k) \cdot f$ is
cocharacter-closed. Conversely, suppose that $W_0$ contains a proper
$f$-stable subspace $U$. Let $\lambda$ be a $k$-defined cocharacter such that
$P_\lambda(G)(k)$ is the stabilizer of $U$ in $G(k)$. Then $f'= \lim \lambda(a) \cdot f$
stabilizes a complement to $U$. As $W_0$ is indecomposable,
$f'$ is not $G(k)$-conjugate to $f$, hence $G(k) \cdot f$ is not
cocharacter-closed.
\end{proof}

\begin{rem}
(i). Note that $f$ is a semisimple endomorphism if and only if
$\mu_f$ is separable (that is, square-free in $\overline{k}[T]$),
see \cite[Thm.\ 8.11]{Roman}.
Proposition \ref{prop:cc_for_GL} thus recovers the fact that an endomorphism
$f$ is semisimple if and only if
$G \cdot f$ is cocharacter-closed over $\overline{k}$.
An endomorphism satisfying the equivalent conditions of Proposition \ref{prop:cc_for_GL}
is called \emph{$k$-semisimple} in \cite{dietz}, where an independent proof
of the equivalence of (ii) and (iii) may be found.

(ii). If $k$ is not perfect, there exist irreducible polynomials in $k[T]$
which are not separable. Hence there exist $f \in V(k)$ with $G(k) \cdot f$
cocharacter-closed over $k$ and $G(\overline{k}) \cdot f$ not 
cocharacter-closed over $\overline{k}$. 

(iii). We have used Levi descent and ascent in the proof of Proposition
\ref{prop:cc_for_GL}. In terms of the characterization by $k[T]$-modules,
Levi descent/ascent is just the fact that a module $W = W_1 \oplus W_2$
is semisimple if and only if both summands are.
\end{rem}

All $G=\GL(W)$-orbits on $V=\End(W)$ are separable (as orbit stabilizers are
principal open subsets of the stabilizers in $\gl(W)$). Hence we may apply
Proposition \ref{prop:perfectorseparable} and Remark \ref{rem:perfectorseparable}
and deduce the following results:

\begin{cor}
\label{cor:GL}
Let $k'/k$ be an algebraic field extension and let $f \in V(k)$ as above.
If $G(k') \cdot f$ is cocharacter-closed over $k'$,
then $G(k) \cdot f$ is cocharacter-closed over $k$.
The converse holds provided that $k'/k$ is separable.

Moreover, the preorder $\prec$ is antisymmetric on the set of 
$G(k)$-orbits on $\End(W_0)$.
\end{cor}

\begin{rem}
The first part of the above corollary may also be deduced from the
following two facts, which are even valid for non-algebraic field extensions:
if $\mu_f$ is square-free in $k'[T]$, it is
also square-free in $k[T]$; if $k'/k$ is separable
and $\mu_f$ is square-free in $k[T]$, then
$\mu_f$ is square-free in $k'[T]$
(see \cite[Ch.\ V, \S 15, Cor.\ 1]{BourbakiAlg}).
\end{rem}

We give an example for the failure of Galois ascent for inseparable field
extensions.

\begin{ex}
\label{ex:insepext}
Let $k=\FF_2(t)$ and suppose that $\mu_f=T^{12} + t$.
This polynomial is irreducible in $k[T]$ (e.g., by
Eisenstein's criterion), so in particular
$G(k) \cdot f$ is cocharacter-closed over $k$, by Proposition 
\ref{prop:cc_for_GL}.
Over separable field extensions $k'/k$, $\mu_f$ can only split up to
the point
\begin{align*}
\mu_f = (T^4 + s)(T^4 + \zeta s)(T^4 + \zeta^2 s),
\end{align*}
where $s^3 = t$ and $\zeta \neq 1, \zeta^3 = 1$, and this
expression is still square-free.
In particular, $G(k')\cdot v$ is then still cocharacter-closed over $k'$,
by Proposition \ref{prop:cc_for_GL}.

If $a \in \overline{k}$ satisfies $a^4 = s$, we may further
decompose $\mu_f$ by inseparable field extensions to obtain
\begin{align*}
\mu_f = (T^2 + a^2)^2 (T^2 + \zeta^2 a^2)^2(T^2 + \zeta a^2)^2
= (T + a)^4 (T + \zeta a)^4 (T+ \zeta^2 a)^4.
\end{align*}

If for example $k' = k(a^2)$,
then the factor $(T^2 + a^2)^2 = (T+a)^4$ in $\mu_f$ 
is a maximal prime power in $\mu_f$.
In the primary cyclic decomposition (cf.\ \cite[Thm.\ 7.6]{Roman}) it hence
corresponds to
a $4$-dimensional indecomposable subspace of $k' \otimes W_0$
where the restriction $g$ of $f$ to this subspace
has minimal polynomial $\mu_g = \chi_g = (T+a)^4$.
This implies (see \cite[Thm.\ 7.12]{Roman}) that a matrix representative for $g$
is given by the companion matrix of $(T + a)^4$,
which over $k'$
is conjugate to the matrix
\begin{align*}
\begin{pmatrix}
0 &a^2 & 0 & a^2 \\
1 & 0 & 0 & 0 \\
0 & 0 & 0 & a^2 \\
0 & 0 & 1 & 0
\end{pmatrix}.
\end{align*}
Clearly, we may pick a cocharacter which kills the entry $a^2$ in the top
right corner while leaving the other entries untouched.
In particular,
the limit along such a cocharacter takes this matrix to a
matrix with minimal polynomial
$T^2 + a^2$, so the orbit $G(k') \cdot g$ is no longer cocharacter-closed over $k'$.  Hence $G(k')\cdot f$ is not cocharacter-closed over $k'$, by Theorem~\ref{thm:leviascentdescent}(ii).
\end{ex}

\begin{rem}
\label{rem:flags}
(i).
For a cocharacter $\lambda \in Y_k(G)$ and $f \in V(k)$, the limit
$\lim \lambda(a) \cdot f$ exists if and only if $f$ stabilizes the flag associated to $\lambda$
in $W_0$. Moreover, after choosing a basis compatible with the flag, the limit $f'$ is then
the endomorphism corresponding to
the block diagonal of the corresponding matrix of $f$.
In particular, $\chi_f = \chi_{f'}$ and $\mu_{f'}$ divides $\mu_f$.

(ii).
As noted before Corollary \ref{cor:GL} above, the assumptions
of Proposition \ref{prop:perfectorseparable} are satisfied
for all $f \in V(k)$, hence the assertion of
Theorem \ref{thm:MCkdefinedstabilizer} holds
in our setup:
if $f$ stabilizes a flag associated to $\lambda$ as in (i), and if
$f'$ is $G$-conjugate to $f$, then it is $R_u(P_\lambda)(k)$ conjugate.
After again choosing a basis compatible with the flag, this translates into the following non-elementary
statement about matrices:
a block upper triangular matrix $A$ over $k$ that is conjugate (over $\overline{k}$) to
the corresponding block diagonal matrix $A'$, can be conjugated to $A'$ by an
block upper triangular matrix over $k$ with identity matrices on the diagonal.
\end{rem}

We use these facts to give an example motivated by classical invariant theory (see \cite{kraft}), concerning cocharacter-closures of sets.

\begin{ex}
\label{ex:charpol}
For $\chi \in k[T]$ consider the set $C_\chi \subset V(k)$ of endomorphisms with prescribed
characteristic polynomial $\chi$. By Remark \ref{rem:flags}(i),
$C_\chi$ is cocharacter-closed.
Moreover, $f \in C_\chi$ has a cocharacter-closed orbit precisely when
$\mu_f$ is the product of all distinct irreducible factors of $\chi$ (Proposition~\ref{prop:cc_for_GL}).
In this case, $\mu_f$ together with $\chi_f= \chi$ uniquely determine the orbit $G(k)\cdot f$.
Indeed, it follows---e.g., from the primary cyclic decomposition
of $W_0$ (cf.\ \cite[Thm.\ 7.2]{Roman})---that 
the elementary divisors of $f$ must in this case be given by the irreducible factors of $\mu$, counted with their multiplicities as factors of $\chi$.
Since the elementary divisors determine a rational canonical form of $f$
(cf.\ \cite[Thm.\ 7.14]{Roman}), they determine the orbit
$G(k)\cdot f$.
In particular, $C_\chi$
contains a unique cocharacter-closed orbit.

It is well known that the closed orbits in $V$---that is, the conjugacy classes of diagonalisable matrices---are parametrized by $\ovl{k}^n$, where $n:= \dim W$.  We now see that the cocharacter-closed orbits in $V(k)$ are parametrized by the space $k^n$.  For let $\pi: V(k) \rightarrow k^n$ be the surjective function that maps $f$ to the coefficients of $1,T,\dots,T^{n-1}$ in $\chi_f$.  The above discussion shows that each fibre of $\pi$ contains a unique orbit that is cocharacter-closed over $k$.
\end{ex}

We finish this section by discussing representations of algebras.  As in the case $k=\mathbb{C}$ (see \cite{kraft}), the
setup in this section may be generalized by replacing the algebra $k[T]$
with some finitely generated $k$-algebra $A_0$.  Let
$A = \overline{k} \otimes_k A_0$, where $A_0$ is a finitely generated $k$-algebra,
and let $V = {\rm mod}_{A,W}$ be the variety of all $A$-module structures on $W$.
We may endow $V$ with a $k$-structure such that $V(k) = {\rm mod}_{A_0,W_0}$.
More specifically, if $A_0$ has $r$ generators, we may identify $V$ with
a subvariety of $\End(W)^r$, with $G:= {\rm GL}(W)$ acting by simultaneous conjugation.  Here instead of considering orbits of a single
endomorphism, we are now studying orbits of tuples of endomorphisms, and these orbits correspond to the isomorphism classes
of $A$-module structures.
It can be shown that the cocharacter-closed orbits (over $k$) in $V(k)$ correspond to isomorphism classes of
semisimple $A_0$-module structures on $W_0$.  Note that all ${\rm GL}(W)$-orbits in $W^n$ are separable, hence the preorder $\prec$ is antisymmetric (cf.\ Corollary~\ref{cor:GL}).  Moreover, by applying Corollary~\ref{cor:MCkdefinedstabilizer}(ii), we get a geometric proof of the algebraic fact that an $A_0$-algebra is semisimple if it becomes semisimple after extending scalars from $k$ to $\ovl{k}$.

Suppose $k$ is algebraically closed.  We say that one module is a {\em degeneration} of another if the tuple defining the first module is in the closure of the ${\rm GL}(W)$-orbit of the tuple defining the second.  This yields a partial ordering $\leq$ on the set of modules.  Zwara gave an algebraic characterization of when a module is a degeneration of another \cite[Thm.\ 1]{zwara}.  There is no obvious geometric notion of degeneration if $k$ is not algebraically closed.  Our formalism of cocharacter-closures and accessibility gives a way to approach this problem for arbitrary fields: our accessibility relation $\prec$ gives a partial order on the set of modules.  This does not coincide with the partial ordering $\leq$ when $k= \ovl{k}$, for the same reason that the Zariski closure of an orbit and its cocharacter-closure over $\ovl{k}$ are not the same.  It would be interesting to take the algebraic condition from \cite[Thm.\ 1]{zwara} and give a geometric interpretation when $k\neq \ovl{k}$.

\section{Further examples}
\label{sec:ex}

Already for $k$ algebraically closed, the notions of cocharacter-closure and
Zariski closure differ, as illustrated by our first example.
More precisely, there may exist non-accessible orbits in the Zariski closure of an orbit.
The fact that not every orbit in an orbit closure
need be $1$-accessible was already noted by Kraft in \cite[II, Rem.\ 4.6]{kraft}.

\begin{ex}
\label{ex:g2}
We consider the unipotent classes in the simple group of type $G_2$
over an algebraically closed field $k$
of characteristic $p$.
Corresponding to each choice of maximal torus $T$ of $G$ and Borel subgroup
$B$ of $G$ containing $T$, we have two simple roots $\alpha$ (short) and $\beta$ (long)
and then for each root $\gamma$ in the root system of $G$ we have a coroot $\gamma^\vee$, a root group $U_\gamma$
and a corresponding root group homomorphism $u_\gamma:k \to U_\gamma$.
We have the following information about conjugacy classes of unipotent elements
(see \cite[Table 22.1.5]{liebeckseitz:2012}, and \cite{stuhler} for the representatives):

\smallskip
\begin{center}
\begin{tabular}{|c|c|c|}
\hline
Class label & Representative & Centralizer \\
\hline
$G_2$         & $u_\alpha(1)u_\beta(1)$ & $2$ \\
$G_2(a_1)$    & $u_\beta(1)u_{2\alpha+\beta}(1)$ & $4$ \\
$(\tilde A_1)_3$ ($p=3$) & $u_\beta(1) u_{\alpha+\beta}(1)$ & $6$ \\
$\tilde A_1$  & $u_\alpha(1)$  & $A_1+3$ ($p\neq 3$), $A_1 + 5$ ($p=3$) \\
$A_1$         & $u_\beta(1)$ & $A_1+5$ \\
Trivial       & $1$                     & $G_2$ \\
\hline
\end{tabular}
\end{center}
\smallskip

The notation in the final column gives the reductive part of the centralizer plus a number denoting the
dimension of the unipotent radical of the centralizer.

\begin{figure}
\label{fig:orbits}
\centering
\begin{tikzpicture}
\matrix (m) [matrix of math nodes, row sep=1.2em, column sep=3em,
text height=1.5ex, text depth=0.25ex]
    {G_2 \\
     G_2(a_1)\\
     \tilde A_1 \\
     A_1 \\
     1 \\
     (p\neq 3)
    \\};
\path (m-1-1) edge (m-2-1);
\path (m-2-1) edge (m-3-1);
\path (m-3-1) edge (m-4-1);
\path (m-4-1) edge (m-5-1);
\end{tikzpicture}
\begin{tikzpicture}
\matrix (m) [matrix of math nodes, row sep=1.2em, column sep=-0.3em,
text height=1.5ex, text depth=0.25ex]
    {G_2 \\
     && G_2(a_1) \\
     & \tilde A_1 \\
     & A_1 \\
     & 1 \\
     & (p \neq 3)
    \\};
\path[->] (m-1-1) edge [bend right=25] (m-3-2);
\path[->] (m-1-1) edge [bend right=25] (m-4-2);
\path[->] (m-1-1) edge [bend right=25] (m-5-2);
\path[->] (m-2-3) edge [bend left=25] (m-3-2);
\path[->] (m-2-3) edge [bend left=25] (m-4-2);
\path[->] (m-2-3) edge [bend left=25] (m-5-2);
\path[->] (m-3-2) edge [bend left=20] (m-5-2);
\path[->] (m-3-2) edge (m-4-2);
\path[->] (m-4-2) edge (m-5-2);
\end{tikzpicture}
\begin{tikzpicture}
\matrix (m) [matrix of math nodes, row sep=1.2em, column sep=-0.3em,
text height=1.5ex, text depth=0.25ex]
    {& G_2 \\
     & G_2(a_1)\\
     & (\tilde A_1)_3 \\
     \tilde A_1\\
     & & A_1 \\
     & 1 \\
     & (p=3)
    \\};
\path (m-1-2) edge (m-2-2);
\path (m-2-2) edge (m-3-2);
\path (m-3-2) edge (m-4-1) edge (m-5-3);
\path (m-4-1) edge (m-6-2);
\path (m-5-3) edge (m-6-2);
\end{tikzpicture}
\begin{tikzpicture}
\matrix (m) [matrix of math nodes, row sep=1.2em, column sep=-0.3em,
text height=1.5ex, text depth=0.25ex]
    { G_2 \\
    && G_2(a_1)\\
    &&&& (\tilde A_1)_3 \\
    & \tilde A_1\\
    &&& A_1 \\
    && 1 \\
    && (p=3)
    \\};
\path[->] (m-1-1) edge (m-4-2);
\path[->] (m-4-2) edge (m-6-3);
\path[->] (m-3-5) edge (m-5-4);
\path[->] (m-5-4) edge (m-6-3);
\path[->] (m-2-3) edge (m-4-2) edge (m-5-4) edge (m-6-3);
\path[->] (m-1-1) edge [bend right=10] (m-5-4);
\path[->] (m-1-1) edge [bend right=35] (m-6-3);
\path[->] (m-3-5) edge [bend left=35] (m-6-3);
\path[->] (m-3-5) edge [bend left=10] (m-4-2);
\end{tikzpicture}
\caption{Closure relation and $1$-accessibility of unipotent classes in $G_2$}
\end{figure}
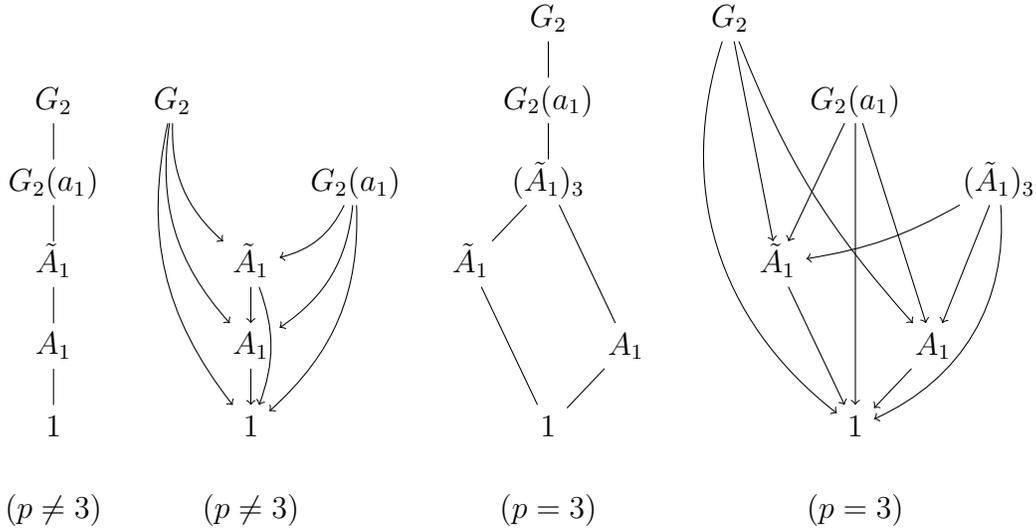

We record the closure and 1-accessibility relations in Figure \ref{fig:orbits}.
Here the first and third diagrams
give the Hasse diagrams for the partial order of containment
between orbit closures taken from 
\cite[II, Prop.\ 10.4]{spalt}.  
In particular, the (Zariski) closure of a given conjugacy class consists
of the union of that class together with those classes below it which are linked to it by a downwards path in the diagram.
The second and fourth diagrams depict the directed graphs corresponding to the
reflexive relation given by $1$-accessibility
(see Definition \ref{def:accessible}).
In this case, it turns out that $1$-accessibility is transitive
(see Example \ref{ex:fromf4} for an instance where this property fails).
The cocharacter-closure of a given class can be read off similarly from these diagrams.
In particular, it can be seen here that the cocharacter-closure
of an orbit may be strictly smaller
than the closure: e.g., the
regular unipotent class labelled $G_2$ has
the subregular class $G_2(a_1)$ in its closure, but not in its cocharacter-closure.

The correctness of the two graphs for $1$-accessibility can be verified by an argument along the following lines.
First, orbits that are not in the closure cannot be $1$-accessible. Therefore, for $p=3$,
the class $A_1$ is not $1$-accessible from $\tilde A_1$ (and vice versa for all $p$).
Second, orbits with a unipotent centralizer cannot be $1$-accessible from other orbits
(as a limit along a cocharacter
is fixed by that cocharacter).
Thus the classes $G_2(a_1)$ and $(\tilde A_1)_3$ are not $1$-accessible
from any other class.
It remains to check that all further possible $1$-accessibility relations do hold.
The trivial class is clearly $1$-accessible from every unipotent class.
For the regular class $G_2$, the representative $u_\alpha(1)u_\beta(1)$ is
destabilized by $(3\alpha+2\beta)^\vee$ to $u_\alpha(1)$, and by
$(2\alpha+\beta)^\vee$ to $u_\beta(1)$. The element
$u_\beta(1)u_{2\alpha+\beta}(1)$ in $G_2(a_1)$ is also taken by
$(2\alpha+\beta)^\vee$ to $u_\beta(1)$, and by $\beta^\vee$
to $u_{2\alpha+\beta}(1)$, which is another representative for $\tilde A_1$.
The representative $u_\alpha(1)$ of $\tilde A_1$ is conjugate by
$u_{2\alpha+\beta}(1)$ to $u_\alpha(1)u_{3\alpha+\beta}(\pm 3)$, and the latter
element is destabilized by $-(\alpha+\beta)^\vee$ to
$u_{3\alpha+\beta}(\pm 3)$. If $p\neq 3$, this is a representative for
the class labelled $A_1$.
Finally, the element $u_\beta(1)u_{\alpha+\beta}(1)$ in $(\tilde A_1)_3$ is destabilized by
$(2\alpha+\beta)^\vee$ to $u_\beta(1)$, and by
$-(3\alpha+\beta)^\vee$ to $u_{\alpha+\beta}(1)$, which represents
$\tilde A_1$.
\end{ex}

\begin{ex}
\label{ex:fromf4}
We give an elementary example which shows that $1$-accessibility is not a transitive relation.
Start with $H=\SL_2(k)$, where $k$ is an algebraically closed field, 
and let $E$ denote the natural module for $H$, with standard basis $\{e_1,e_2\}$.
Set $W = S^2(E)$, the symmetric square of $E$, and let $\{x^2,xy,y^2\}$ denote the standard basis for $W$.
Let $\lambda$ denote the diagonal cocharacter of $H$ which acts with weight $1$ on $e_1$ and weight $-1$ on $e_2$,
and with weights $2$, $0$ and $-2$ on $x^2$, $xy$ and $y^2$ respectively.

Now let $G:= H\times k^*$ and let $\mu:k^*\to G$ be the cocharacter given by $\mu(a) = (1,a)\in G$ for each $a\in k^*$.
Then the images of $\lambda$ and $\mu$ generate a maximal torus $T$ of $G$,
and $\lambda$ and $\mu$ span the cocharacter group $Y_T$.
Let the $k^*$-factor of $G$ act on $E$ with weight $-1$ and on $W$ with weight $2$.
Since these actions of $k^*$ commute with the $H$-actions,
we get actions of $G$ on $E$ and $W$, and we can combine these to get
an action of $G$ on $V:=W \oplus E$.

Consider the element $v:=xy+e_1 \in V$.
Then $\lambda(a)\cdot v = xy + ae_1$, so $\lim_{a\to 0} \lambda(a)\cdot v$ exists and equals $v':=xy$.
Now let
$u=(u_h,1)$ in $G$, where $u_h\in R_u(P_\lambda)(H)$ is the unipotent element for which $u_h\cdot xy = x^2+xy$.
Then $u\cdot v' = x^2+xy$.
Define another cocharacter $\sigma \in Y(T)$ by $\sigma = \mu-\lambda$.
We have $\sigma(a)\cdot(x^2+xy) = x^2 + a^2xy$, so $\lim_{a\to 0} \sigma(a)\cdot(x^2+xy) = x^2$.
Set $v'':=x^2$.
The above shows that, in the language of this paper, we have a sequence of orbits
$G(k) \cdot v$, $G(k) \cdot v'$, $G(k) \cdot v''$, with
$G(k) \cdot v'$ $1$-accessible from $G(k) \cdot v$ and
$G(k) \cdot v''$ $1$-accessible from $G(k) \cdot v'$.
One can see by direct calculation that these are distinct orbits.

We claim that $G(k) \cdot v''$ is not $1$-accessible from $G(k) \cdot v$, which shows that $1$-accessibility is not transitive.
Since $\lambda$ and $\mu$ generate the cocharacter group $Y(T)$ of the maximal torus $T$ of $G$,
and every cocharacter of $G$ is conjugate to one in $Y(T)$,
to check that $G(k) \cdot v''$ is not $1$-accessible from $G(k) \cdot v$, it suffices to show that
$\lim_{a\to 0} (m\lambda+n\mu)(a)\cdot(g\cdot v) \not\in G(k) \cdot v''$ for any $g \in G$ and $m,n \in \ZZ$
for which the limit exists.
Let $g = (h,b)\in G$ be an arbitrary element, and suppose $h$ has matrix
$\left(\begin{array}{cc} q&r\\s&t \end{array}\right)$ with respect to the fixed basis for $E$.
Then we can write
\begin{align*}
(m\lambda+n\mu)(a)\cdot(g\cdot v)
&=
(m\lambda+n\mu)(a)\cdot
(b^2(qrx^2 + (qt+rs)xy + sty^2) + b^{-1}(qe_1 + se_2))\\
&=
b^2(qra^{2m+2n}x^2 + (qt+rs)a^{2n}xy + sta^{-2m+2n}y^2) \\
&\phantom{{}=b^2(qra^{2m+2n}x^2 + (qt+rs)a} + b\inverse(qa^{m-n}e_1 + sa^{-m-n}e_2).
\end{align*}
Now any $G$-conjugate of $v''=x^2$ has zero component in $E$,
hence to stand any chance of the limit existing and lying in $G(k) \cdot v''$
we need to kill off the $e_1$ and $e_2$ component in the limit.
For the element $h$ to lie in $\SL_2(k)$, we can't have $q=s=0$, so we have three cases to consider.

First suppose $q\neq0$ and $s\neq 0$.
Then we need $m-n>0$ and $-m-n>0$ to kill off the $e_1$ and $e_2$ components.
But $-m-n>0$ implies $2m+2n<0$, so looking at the coefficient of $x^2$, we must have $qr=0$ and hence $r=0$.
Also, $m-n>0$ implies $-2m+2n<0$, so looking at the coefficient of $y^2$ we must have $st=0$ and hence $t=0$.
But then $h$ is not an invertible matrix, so this is impossible.

Second suppose $q=0$ and $s\neq 0$.
Then we need $-m-n>0$, i.e., $-m>n$.
For $h$ to be invertible, we must have $r\neq 0$ and
then the coefficient of $xy$ tells us that $n\geq 0$.
The inequality $-m>n$ now gives $-m>0$, and hence $-2m+2n>0$.
In this case the limit does exist and it equals $0$ if $n>0$ or $b^2rsxy$ if $n=0$.
In the first case, this is clearly not conjugate to $v''$, and in the second case the limit
is conjugate to $v'=xy$ (e.g., by the element $\left(I,(b^2rs)^{\frac{1}{2}}\right)\in G$).
As $v'\not\in G\cdot v''$, we therefore cannot get into the orbit of $v''$ when $q=0$ and $s\neq0$.

The case $q\neq0$ and $s=0$ is similar---the limit, when it exists, is either $0$ or conjugate to $v'$---and we see that $G(k) \cdot v''$ is not $1$-accessible from $G(k) \cdot v$, as claimed.
\end{ex}


\bigskip
{\bf Acknowledgments}:
The authors acknowledge the financial support of EPSRC Grant EP/L005328/1 and 
of Marsden Grants UOC1009 and UOA1021.
Part of the research for this paper was carried out while the
authors were staying at the Mathematical Research Institute
Oberwolfach supported by the ``Research in Pairs'' programme.
Also, part of this paper was written during mutual visits to Auckland, 
Bochum and York.  We are grateful to the referees for their careful reading of the paper and for helpful suggestions.


\end{document}